\documentclass[a4paper, 10pt]{amsart}
\usepackage{amsmath, amscd}
\usepackage{amssymb, color, hyperref}

\usepackage[fontsize=11pt]{scrextend}

\theoremstyle{plain}
\newtheorem{theorem}{\bf Theorem}[section]
\newtheorem{proposition}[theorem]{\bf Proposition}
\newtheorem{lemma}[theorem]{\bf Lemma}

\theoremstyle{definition}
\newtheorem{example}[theorem]{\bf Example}

\newcommand{\N}{\mathbb N}
\newcommand{\Z}{\mathbb Z}
\newcommand{\R}{\mathbb R}
\newcommand{\Q}{\mathbb Q}
\newcommand{\F}{\mathbb F}

\renewcommand{\t}{\, | \,}

\newcommand{\BF}{\text{\rm BF}}

 \DeclareMathOperator{\ord}{ord}

\DeclareMathOperator{\spec}{spec} \DeclareMathOperator{\supp}{supp}
\DeclareMathOperator{\Pic}{Pic} 
 
\DeclareMathOperator{\End}{End} 
 \DeclareMathOperator{\Ca}{\mathsf {Ca}}
\DeclareMathOperator{\Ta}{\mathsf {Ta}}

\newcommand{\red}{{\text{\rm red}}}

\numberwithin{equation}{section}

\renewcommand{\time}{\negthinspace \times \negthinspace}

\makeatletter
\newcommand\zeu@Scale{1.25}
\makeatother

\begin{document}

\author{Alfred Geroldinger and Qinghai Zhong}

\address{University of Graz, NAWI Graz \\
Institute for Mathematics and Scientific Computing \\
Heinrichstra{\ss}e 36\\
8010 Graz, Austria}

\email{alfred.geroldinger@uni-graz.at, qinghai.zhong@uni-graz.at}
\urladdr{http://imsc.uni-graz.at/geroldinger, http://qinghai-zhong.weebly.com/}

\keywords{Krull monoids, bounded hereditary prime rings, maximal orders, sets of lengths, sets of distances, elasticities, catenary degrees}

\subjclass[2010]{20M13, 13A05, 13F05, 16H10, 16U30}

\thanks{This work was supported by the Austrian Science Fund FWF, Project Number P28864-N35}

\begin{abstract}
Transfer Krull monoids are a recent concept including all commutative Krull domains and also, for example, wide classes of non-commutative Dedekind domains. We show that transfer Krull monoids are fully elastic (i.e., every rational number between $1$ and the elasticity of the monoid can be realized as the elasticity of an element). In commutative Krull monoids which have sufficiently many prime divisors in all classes of their class group, the set of catenary degrees and the set of tame degrees are intervals. Without the assumption on the distribution of prime divisors, arbitrary finite sets can be realized as sets of catenary degrees and as sets of tame degrees.
\end{abstract}

\title[Sets of Arithmetical Invariants]{Sets of Arithmetical Invariants \\ in Transfer Krull Monoids}
\maketitle

\medskip
\section{Introduction} \label{1}
\medskip

A transfer Krull monoid is a monoid having a weak transfer homomorphism to a commutative Krull monoid or, equivalently, to a monoid of zero-sum sequences. A successful strategy to study the arithmetic of a transfer Krull monoid $H$ runs as follows: first, study the  arithmetic of a monoid of zero-sum sequences $B$  with methods from additive combinatorics and then pull back arithmetic properties from $B$ to the monoid $H$ with the help of the transfer homomorphism. By definition, commutative Krull domains are transfer Krull but, in order to mention a non-commutative example,  also wide classes of hereditary noetherian prime rings turned out to be  transfer Krull. Our main results are formulated in the abstract setting of transfer Krull monoids. The objects we have in mind are discussed in Subsection \ref{sec-2.5} and in Example \ref{4.d}.

Let $H$ be a cancellative monoid and, for simplicity of discussion, suppose that $H$ is commutative.  The ascending chain condition on principal ideals of $H$ guarantees that every element of $H$ can be written as a finite product of irreducible elements. But in general such a factorization need not be unique. Indeed, all factorizations are unique (i.e., the monoid $H$ is factorial) if and only if $H$ is a commutative Krull monoid with trivial class group. Arithmetical invariants, such as elasticities, catenary and tame degrees, describe the non-uniqueness of factorizations. For an element $a \in H$, the elasticity of $a$ is the supremum of  $\ell/k$ over all $k, \ell$ for which there are factorizations of the form $a=u_1 \cdot \ldots \cdot u_k=v_1 \cdot \ldots \cdot v_{\ell}$, where all $u_i$ and $v_j$ are irreducibles. The catenary degree $\mathsf c (a)$ of $a$ is the smallest integer $N$ with the following property: for each two factorizations $z, z'$ of $a$, there exist factorizations $z=z_0, \ldots, z_k=z'$ of $a$ such that, for each $i \in [1,k]$, $z_i$ arises from $z_{i-1}$ by replacing at most $N$ atoms from $z_{i-1}$ by at most $N$ new atoms. The elasticity $\rho (H)$ of $H$ is the supremum over all $\rho (a)$ and the catenary degree $\mathsf c (H)$ of $H$ is the supremum over all $\mathsf c (a)$. By definition, we have $\mathsf c (H)=0$ if and only if $H$ is factorial and if this holds, then $\rho (H)=1$. The elasticity and the catenary degree are classical invariants in factorization theory but only in the last couple of years the full set of elasticities $\{ \rho (a) \mid a \in H \}$ and the set of all catenary degrees $\{ \mathsf c (a) \mid a \in H\}$ found attention in the literature. Their study is the goal of the present paper.

In Section \ref{3} we show that in a transfer Krull monoid $H$ every rational number lying between $1$ and $\rho (H)$ can be realized as the elasticity of an element $a \in H$ (Theorem \ref{3.1}). In Section \ref{4} we study sets of catenary degrees and related sets of distances and of minimal relations, and in Section \ref{5} we study the set of tame degrees. The main results are Theorems \ref{4.b} and \ref{5.2}. Roughly speaking, these results say that, if a commutative Krull monoid has sufficiently many prime divisors in all classes, then all sets of invariants under consideration are intervals. It is comparatively easy to show that without any assumption on the distribution of prime divisors any finite set can occur as any of these sets of invariants (Propositions \ref{4.c} and \ref{5.5}).

\medskip
\section{Background on the arithmetic of transfer Krull monoids} \label{2}
\medskip

\noindent
\subsection{\bf Notation.} \label{sec-2.1}
We denote by $\N$ the set of positive integers and set $\N_0 = \N \cup \{0\}$. If $a, b \in \R$, we write $[a, b] = \{ x \in \Z \mid a \le x \le b\}$ for the discrete interval from $a$ to $b$. Let $A, B \subset \Z$ be subsets of the integers. Then $A+B = \{a+b \mid a \in A, b \in B\}$ denotes their sumset and for $m \in \Z$ we set $m+A = \{m\}+A$. If $A = \{a_1, \ldots, a_k\}$ with $k \in \N_0$ and $a_1 < \ldots < a_k$, then $\Delta (A)= \{a_{\nu+1} - a_{\nu} \mid \nu \in [1, k-1]\}$ is the set of distances of $A$. If $A \subset \N$ is nonempty, then $\rho (A) = \sup A / \min A \in \Q_{\ge 1} \cup \{\infty\}$ is the elasticity of $A$ and if $A = \{0\}$, then $\rho (A)=1$. For every $n \in \N$, we denote by $C_n$ a cyclic group of order $n$.

\noindent
\subsection{\bf Atomic monoids and sets of lengths.} \label{sec-2.2}
By a {\it monoid}, we mean a left and right cancellative semigroup with identity element and all monoid homomorphisms are assumed to respect the identity element.
For a ring $R$, we denote by  $R^{\bullet}$  the monoid of regular elements of $R$.
Let $H$ be a multiplicatively written monoid. We denote by $H^{\times}$ the group of units of $H$ and we say that $H$ is reduced if $H^{\times} = \{1_H\}$. An element $a \in H$ is {\it irreducible} (an {\it atom}) if $a \notin H^{\times}$ and if, for all $b, c \in H$, $a=bc$ implies that $b \in H^{\times}$ or $c \in H^{\times}$. We denote by $\mathcal A (H)$ the set of atoms of $H$. The monoid $H$ is called {\it atomic} if every noninvertible element can be written as a finite product of atoms of $H$. If $a  = u_1 \cdot \ldots \cdot u_k \in H$, where $k \in \N$ and $u_1, \ldots, u_k \in \mathcal A (H)$, then $k$ is called a {\it factorization length} of $a$ and
\[
\mathsf L (a) = \{ k \in \N \mid a \ \text{has a factorization of length $k$} \} \subset \N
\]
is called the {\it set of lengths} of $a$. For convenience we set $\mathsf L (a)=\{0\}$ for all $a \in H^{\times}$. For two elements $a , b \in H$, we have $\mathsf L (a) + \mathsf L (b) \subset \mathsf L (ab)$ and, clearly, $\mathsf L (a) = \{1\}$ if and only if $a$ is an atom. We call
\[
\mathcal L (H) = \{ \mathsf L (a) \mid a \in H \}
\]
the {\it system of sets of lengths} of $H$. The monoid $H$ is said to be
\begin{itemize}
\item {\it half-factorial} if $H$ is atomic and $|L|=1$ for all $L \in \mathcal L (H)$,
\item a \BF-{\it monoid} if $H$ is atomic and all $L \in \mathcal L (H)$ are finite.
\end{itemize}

\smallskip
\noindent
\subsection{\bf Commutative Krull monoids.} \label{sec-2.3}
For a set $P$, we denote by $F = \mathcal F (P)$ the free abelian monoid with basis $P$. Then every $a \in F$ has a unique representation in the form
\[
a = \prod_{p \in P} p^{\mathsf v_p (a)} \,,
\]
where $\mathsf v_p \colon F \to \N_0$ denotes the $p$-adic exponent. We call $\supp (a) = \{ p \in P \mid \mathsf v_p (a) > 0 \} \subset P$ the support of $a$ and $|a|_F = |a| = \sum_{p \in P} \mathsf v_p (a) \in \N_0$ the length of $a$. We gather the basics on commutative Krull monoids (detailed presentations of the theory of Krull monoids are given in \cite{HK98} and \cite{Ge-HK06a}).  Let $H$  be a commutative  monoid. Then $H_{\red} = H/H^{\times}$ denotes the associated reduced monoid of $H$ and $\mathsf q (H)$ the quotient group of $H$.
A monoid homomorphism $\varphi \colon H \to D$ to a commutative monoid $D$ is called
\begin{itemize}
\item a {\it divisor homomorphism} if $\varphi (a) \t \varphi (b)$ implies that $a \t b$ for all $a, b \in H$,

\item a {\it divisor theory} (for $H$) if $\varphi$ is a divisor homomorphism, $D$ is free abelian, and for every $\alpha \in D$ there are $a_1, \ldots, a_m \in H$ such that $\alpha = \gcd ( \varphi (a_1), \ldots, \varphi (a_m) )$.
\end{itemize}
The monoid $H$ is a {\it Krull monoid} if it satisfies one of the following equivalent conditions (\cite[Theorem 2.4.8]{Ge-HK06a}){\rm \,:}
\begin{itemize}
\item[(a)] $H$ is completely integrally closed and $v$-noetherian.
\item[(b)] $H$ has a divisor theory.
\item[(c)] There is a divisor homomorphism from $H$ to a factorial monoid.
\end{itemize}

Suppose that $H$ is a commutative Krull monoid. Then there is a free abelian monoid $F = \mathcal F (P)$ such that the inclusion $H_{\red} \hookrightarrow F$ is a divisor theory. This implies immediately that $H$ is a \BF-monoid. The group $\mathcal C (H) = \mathsf q (F)/\mathsf q (H_{\red})$ is the (divisor) class group of $H$ and $G_P = \{ [p] = p \mathsf q (H_{\red}) \mid p \in P \} \subset \mathcal C (H)$ is the set of classes containing prime divisors. A commutative monoid is factorial if and only if it is Krull with trivial class group.

Next we discuss a Krull monoid with a combinatorial flavour which plays a universal role in the arithmetic theory of Krull monoids. Let $G$ be an additively written abelian group and $G_0 \subset G$ a subset. In additive combinatorics, a {\it sequence} over $G_0$ means a finite unordered sequence of terms from $G_0$ with repetition being allowed. As usual, we consider sequences as elements of the free abelian monoid over $G_0$ whence let
\[
S = g_1 \cdot \ldots \cdot g_{\ell} = \prod_{g \in G_0} g^{\mathsf v_g (S)} \in \mathcal F (G_0)
\]
be a sequence over $G_0$. Then $|S|=\ell \in \N_0$ is its length and $\sigma (S)=g_1 + \ldots + g_{\ell} \in G$ is its sum, and we set $-S = (-g_1) \cdot \ldots \cdot (-g_{\ell})$. The monoid
\[
\mathcal B (G_0) = \{ S \in \mathcal F (G_0) \mid \sigma (S)=0 \} \subset \mathcal F (G_0)
\]
denotes the {\it monoid of zero-sum sequences}, and since the inclusion $\mathcal B (G_0) \hookrightarrow \mathcal F (G_0)$ is a divisor homomorphism, $\mathcal B (G_0)$ is a Krull monoid by Property (c). The atoms of $\mathcal B (G_0)$ are precisely the minimal zero-sum sequences over $G_0$ and
\[
\mathsf D (G_0) = \sup \{ |A| \mid A \ \text{is a minimal zero-sum sequence over $G_0$} \} \in \N_0 \cup \{\infty\}
\]
denotes the {\it Davenport constant} of $G_0$. If $G_0$ is finite, then $\mathcal B (G_0)$ is finitely generated, $\mathcal A (G_0)$ is finite, and $\mathsf D (G_0)< \infty$ (\cite[Theorem 3.4.2]{Ge-HK06a}). Suppose that $G \cong C_{n_1} \oplus \ldots \oplus C_{n_r}$, where $r = \mathsf r (G)$ is the rank of $G$ and $n_1, \ldots, n_r \in \N$ with $1 < n_1 \mid \ldots \mid n_r$.
Then
\begin{equation} \label{Davenport1}
\mathsf D^* (G):= 1 + \sum_{i=1}^r (n_i-1) \le \mathsf D (G) \,,
\end{equation}
and equality holds, among others, for $p$-groups and groups with rank at most two (\cite[Theorems 5.5.9 and 5.8.3]{Ge-HK06a}).
In particular, we see that $\mathsf D (G)=|G|$ for $|G|\le 2$ and that
\begin{equation} \label{Davenport2}
\mathsf D (G)=3 \quad \text{ if and only if} \quad  G \cong C_3 \quad \text{ or} \quad  G \cong C_2 \oplus C_2 \,.
\end{equation}

\smallskip
\noindent
\subsection{\bf Transfer Krull monoids.} \label{sec-2.4}

Let $H$ and $B$ be atomic monoids and  $\theta \colon H \to B$ a homomorphism. We consider  the following properties:
\begin{itemize}
\item[{\bf (T1)}] $B = B^{\times} \theta (H) B^{\times}$ and $\theta^{-1} (B^{\times})=H^{\times}$.

\item[{\bf (T2)}] If $a \in H$, $b_1, b_2 \in B$, and $\theta (a)=b_1b_2$, then there exist $a_1, a_2 \in H$ and $\varepsilon \in B^{\times}$ such that $a=a_1a_2$, $\theta (a_1) = b_1 \epsilon^{-1}$, and $\theta (a_2) = \varepsilon b_2$.

\item[{\bf (WT2)}] If $a \in H$, $n \in \N$, $v_1, \ldots, v_n \in \mathcal A (B)$ and $\theta (a) = v_1 \cdot \ldots \cdot v_n$, then there exist $u_1, \ldots, u_n \in \mathcal A (H)$ and a permutation $\tau \in \mathfrak S_n$ such that $a = u_1 \cdot \ldots \cdot u_n$ and $\theta (u_i) \in B^{\times} v_{\tau (i)} B^{\times}$ for each $i \in [1,n]$.
\end{itemize}
The map $\theta$ is called a {\it  transfer homomorphism} (resp. a {\it weak transfer homomorphism}) if it satisfies {\bf (T1)} and {\bf (T2)} (resp. {\bf (T1)} and {\bf (WT2)}). Every transfer homomorphism is a weak transfer homomorphism and the converse holds if $H$ and $B$ are both commutative (\cite[Section 2]{Ba-Sm15}; in general, this is not true as can be seen from \cite{Ba-Ba-Go14}).
If $\theta \colon H \to B$ is a weak transfer homomorphism, then it is easy to check (e.g., \cite[Lemma 2.7]{Ba-Sm15}) that
\begin{equation} \label{transfer}
\theta (\mathcal A (H))=\mathcal A (B) \quad \text{ and  } \quad \mathcal L (H)=\mathcal L (B) \,.
\end{equation}
An atomic  monoid $H$ is said to be a {\it transfer Krull monoid} if one of the following two equivalent properties is satisfied:
\begin{itemize}
\item[(a)] There is a commutative Krull monoid $B$ and a weak transfer homomorphism $\theta \colon H \to B$.

\item[(b)] There is an abelian group $G$, a subset $G_0 \subset G$, and a weak transfer homomorphism $\theta \colon H \to \mathcal B (G_0)$.
\end{itemize}
In case (b) we say that $H$ is a transfer Krull monoid over $G_0$. Since $\mathcal B (G_0)$ is a commutative Krull monoid, Condition (b) implies Condition (a). Conversely, since every commutative Krull monoid has a transfer homomorphism to a monoid of zero-sum sequences (\cite[Theorem 3.4.10]{Ge-HK06a}) and since the composition of weak transfer homomorphisms is a weak transfer homomorphism again, Condition (a) implies Condition (b).

\smallskip
\subsection{\bf Examples.} \label{sec-2.5}
We provide a first list of examples of commutative Krull monoids and of transfer Krull monoids (note that transfer Krull monoids need neither be commutative nor completely integrally closed nor $v$-noetherian). A more extended list can be found in the survey \cite{Ge16c}. In Example \ref{4.d} we discuss special features of these and of some more examples.

1. (Commutative ring theory)   A commutative  noetherian domain is a Krull domain if and only if it is integrally closed, and a commutative integral domain is a Krull domain if and only if its multiplicative monoid of nonzero elements is a Krull monoid. This generalizes to rings with zero-divisors. Indeed, in a $v$-Marot ring  the monoid of regular elements   is Krull if and only if the ring is a Krull ring (\cite[Theorem 3.5]{Ge-Ra-Re15c}). For cluster algebras that are Krull we refer to \cite{Ga-La-Sm19a}.

2. (Module theory) Let $R$ be a ring and $\mathcal C$ be a class of right $R$-modules which is closed under finite direct sums, direct summands, and isomorphisms. If $\End_R (M)$ is semilocal for all modules $M$ in $\mathcal C$, then the monoid $\mathcal V (\mathcal C)$ of isomorphism classes of modules in $\mathcal C$ is Krull (see \cite[Theorem 3.4]{Fa02} for the original result and also \cite{Fa-Wi04, Fa06a, Fa12a, Ba-Wi13a}).

3. (Non-commutative ring theory) Let $R$ be a bounded hereditary noetherian prime ring. If every stably free right $R$-ideal is free, then the monoid of regular elements of $R$ is transfer Krull (\cite[Theorem 4.4]{Sm18a}).

\medskip
\section{The set of elasticities} \label{3}
\medskip

Let $H$ be a BF-monoid. Then
\[
\rho (H) = \sup \{ \rho (L) \mid L \in \mathcal L (H) \} \in \R_{\ge 1} \cup \{\infty\}
\]
is the {\it elasticity} of $H$, and this is one of the first invariants studied in  factorization theory.  We say that $H$ has accepted elasticity if there is an $L \in \mathcal L (H)$ such that $\rho (L)=\rho (H)$. Every commutative finitely generated monoid has accepted elasticity (\cite[Theorem 3.1.4]{Ge-HK06a}). If $H$ is a transfer Krull monoid over a subset $G_0$ of an abelian group, then we set (as usual) $\rho (G_0) := \rho \big( \mathcal B (G_0) \big)$ and recall that (\cite[Theorem 3.4.10]{Ge-HK06a})
\[
\rho (H) = \rho (G_0) \le \mathsf D (G_0)/2  \quad \text{ where  equality holds if} \quad  G_0 = -G_0\,.
\]
Furthermore, we say that $G_0$ is half-factorial if the monoid $\mathcal B (G_0)$ is half-factorial. Since the 1990s the elasticity of commutative rings has found wide attention in the literature. The reader may want to consult the work of D.D. Anderson, D.F. Anderson, S. Chapman, P.J. Cahen, J.L. Chabert,  J. Coykendall, F. Halter-Koch, M. Picavet-L'Hermitte, H. Kim, and others. A characterization of when the elasticity of finitely generated domains is finite is given in \cite{Ka05a}. For some  recent results we refer to \cite{An-Ju15a, An-Ju17a, Ba-Co16a, Ka16b}).
In \cite{Ch-Ho-Mo06, B-C-H-M06}, Chapman et al. initiated the study of the set $\{\rho (L) \mid L \in \mathcal L (H) \} \subset \Q_{\ge 1}$ of elasticities of all sets of lengths.
By definition, $H$ is half-factorial if and only if $\rho (H)=1$ and clearly this is equivalent to $\{ \rho (L) \mid L \in \mathcal L (H)\} = \{1\}$. The reverse extremal case, namely when the set of elasticities is as large as possible,  found special attention.
We say that $H$ is {\it fully elastic} if for every rational number $q$ with $1 \le q < \rho (H)$ there is an $L \in \mathcal L (H)$ such that $\rho (L)=q$.

If every finite subset of $\N_{\ge 2}$ is a set of lengths of $H$ (a property which holds true for Krull monoids with infinite class group having prime divisors in all classes \cite[Theorem 7.4.1]{Ge-HK06a} and for the ring of integer-valued polynomials over rings of integers \cite{Fr-Na-Ri18a}), then obviously $H$ is fully elastic. But this very strong property is far from being necessary. In \cite{B-C-C-K-W06} it was proved that every commutative monoid having a prime element is fully elastic.

On the other hand, we know that strongly primary monoids (including one-dimensional local Mori domains and numerical monoids) are not fully elastic (\cite[Theorem 5.5]{Ge-Sc-Zh17b}). The set of elasticities in numerical monoids was recently studied in \cite{Ba-Ne-Pe17a}. Arithmetic congruence monoids which are not fully elastic can be found in the survey \cite{Ba-Ch14a}.

As  the main result of this section we prove that every transfer Krull monoid is fully elastic.

\begin{theorem} \label{3.1}
Every transfer Krull monoid is fully elastic.
\end{theorem}

We proceed in a series of lemmas. The proof of Theorem \ref{3.1} will be given at the end of this section.

\medskip
\begin{lemma}\label{3.2}
Let $n\in \N$ and $a_1,\ldots, a_n\in \N$ be pairwise distinct positive integers. If there exist $x_1,\ldots, x_n\in \N_0$ and $t\ge 2$ such that $a_1x_1+\ldots+a_nx_n=ta_1\ldots a_n$, then there exist $x_i'\in [0,x_i]$, for all $i\in [1,n]$, such that
\[
a_1x_1'+\ldots+a_nx_n'=a_1\ldots a_n\,.
\]
In particular, there exist $x_i^{(j)}\in [0,x_i]$, for all $i\in[1,n]$ and $j\in[1,t]$, such that  $\sum_{j\in [1,t]}x_i^{(j)}=x_i$ for every $i\in [1,n]$ and $\sum_{i\in [1,n]}a_ix_i^{(j)}=a_1\ldots a_n$ for every $j\in [1,t]$.
\end{lemma}

\begin{proof}
The assertion is obvious for $n=1$. If $n=2$, then $a_1x_1\ge a_1a_2$ or $a_2x_2\ge a_1a_2$ whence the assertion follows immediately.

Suppose that $n \ge 3$. After renumbering if necessary we
assume that  $a_1x_1\ge \frac{ta_1\ldots a_n}{n}$. Since $a_1,\ldots, a_n$ are pairwise distinct positive integers, we obtain  $\max\{a_1,\ldots, a_n\}\ge n$ and hence
 $$x_1\ge \frac{2a_2\ldots a_n}{n}\ge \min\{a_2,\ldots, a_n\}\,.$$
For each  $i\in [2,n]$, we set $x_i=y_ia_1+r_i$ with $r_i\in[0, a_1-1]$. Then
\[
a_2r_2+\ldots+a_nr_n\le (a_1-1)(a_2+\ldots +a_n)\le a_1\ldots a_n\,.
\]
Therefore $a_1(x_1+y_2a_2+\ldots +y_na_n)\ge ta_1\ldots a_n-a_1\ldots a_n\ge a_1\ldots a_n$ which implies that
\[
x_1+y_2a_2+\ldots + y_na_n\ge a_2\ldots a_n\,.
\]
If $y_2a_2+\ldots +y_na_n\le a_2\ldots a_n$, then there exists $x_1'\in [0,x_1]$ such that $x_1'+y_2a_2+\ldots +y_na_n=a_2\ldots a_n$ and hence $a_1x_1'+a_2y_2a_1+\ldots +a_ny_na_1=a_1\ldots a_n$. If $y_2a_2+\ldots+ y_na_n>a_2\ldots a_n$, then we can choose $y_i'\in [0,y_i]$ for each $i\in [2,n]$ such that $$0\le a_2\ldots a_n-(y_2'a_2+\ldots +y_n'a_n)\le \min\{a_2,\ldots, a_n\}\le x_1\,.$$
Thus there exists $x_1'\in [0,x_1]$ such that $x_1'+y_2'a_2+\ldots y_n'a_n=a_2\ldots a_n$ and hence $a_1x_1'+a_2y_2'a_1+\ldots+ a_ny_n'a_1=a_1\ldots a_n$.

The in particular statement follows by induction on $t$.
\end{proof}

Let $H$ be a BF-monoid, $a \in H$ with $\rho (\mathsf L (a))=\rho (H)$, and $n \in \N$. Then the $n$-fold sumset $\mathsf L (a) + \ldots + \mathsf L (a)$ is contained in $\mathsf L (a^n)$ whence
\[
\max \mathsf L (a^n) \ge n \max \mathsf L (a) , \quad  \min \mathsf L (a^n) \le n \min \mathsf L (a) \,,
\]
\[
\rho ( H) \ge \rho (\mathsf L (a^n)) \ge \frac{n \max \mathsf L (a)}{n \min \mathsf L (a)} = \rho (\mathsf L (a)) = \rho (H) \,, \quad \text{and thus} \quad \rho (\mathsf L (a^n)) = \rho (H) \,.
\]
We will use this property without further mention.

\medskip
\begin{lemma}\label{3.3}
Let $G$ be an  abelian group,  $G_0 \subset G$ a finite subset,  and $A\in \mathcal B(G_0)$ with $\rho(A)=\rho(G_0)>1$.

\begin{enumerate}
\item There exists an atom $A_0\in \mathcal A(G_0)$ such that $\supp(A_0)\subsetneq \supp(A)$ is half-factorial.

\item There exists an $M\in \N$ satisfying that  for every $k\in \N$ and every $\ell \in \N_0$, there exist $\ell_1,\ldots, \ell_k\in \Z$ with $\ell_1+\ldots+\ell_k=\ell$ and $A^MA_0^{\ell_i} \in \mathcal B (G_0)$ such that
\begin{align*}
\max\mathsf L(A^{Mk}A_0^{\ell})=\sum_{i=1}^{k}\max\mathsf L(A^MA_0^{\ell_i}) \ \  \text{ and } \ \
\min\mathsf L(A^{Mk}A_0^{\ell})=\sum_{i=1}^{k}\min\mathsf L(A^MA_0^{\ell_i})\,.
\end{align*}

\item For every $r\ge M|A|(\mathsf D(G_0)-1)$,  we have
\begin{align*}
\max\mathsf L(A^MA_0^{r})&=\max\mathsf L(A^MA_0^{M|A|(\mathsf D(G_0)-1)})+r-M|A|(\mathsf D(G_0)-1) \quad \text{and} \\
\min\mathsf L(A^MA_0^{r})&=\min\mathsf L(A^MA_0^{M|A|(\mathsf D(G_0)-1)})+r-M|A|(\mathsf D(G_0)-1)\,.
\end{align*}
We  define  $r_M$ to be the minimal non-negative integer such that for all $r\ge r_M$, we have
\begin{align*}
\max\mathsf L(A^MA_0^{r})&=\max\mathsf L(A^MA_0^{r_M})+r-r_M  \quad \text{and}  \\
\min\mathsf L(A^MA_0^{r})&=\min\mathsf L(A^MA_0^{r_M})+r-r_M\,.
\end{align*}
We define $\tau_M$ to be the maximal non-negative integer $\ell$ such that
 $\rho(\mathsf L(A^MA_0^{\ell}))=\rho(\mathsf L(A^M))$.

\item  For every $k,\ell\in \N_0$, there exists an $N = N (k, \ell) \in \N$ such that for every $t\in \N$ we have,
\[
\begin{aligned}
\max\mathsf L((A^{MkN}A_0^{\ell N})^t)&=t\max\mathsf L(A^{MkN}A_0^{\ell N}) \quad \text{and} \\
\min\mathsf L((A^{MkN}A_0^{\ell N})^t)&=t\min\mathsf L(A^{MkN}A_0^{\ell N})\,.
\end{aligned}
\]
Let $N'$ be a multiple of $N$. Then there exist $\ell_1,\ldots, \ell_{kN'}\in \Z$ such that all $A^M A_0^{\ell_i} \in \mathcal B (G)$,  $\ell_1+\ldots+\ell_{kN'}=\ell N'$ and
\begin{align*}
\max\mathsf L(A^{MkN'}A_0^{\ell N'})&=\sum_{i=1}^{kN'}\max\mathsf L(A^MA_0^{\ell_i}) \quad  \text{and} \\ \min\mathsf L(A^{MkN'}A_0^{\ell N'})&=\sum_{i=1}^{kN'}\min\mathsf L(A^MA_0^{\ell_i})\,.
\end{align*}
Moreover, for all $\ell_1,\ldots, \ell_{kN'}\in \Z$ satisfying these properties the following holds:
\begin{enumerate}

\item[(a)] For distinct $i_1,i_2\in [1, kN']$ and any $t_1,t_2\in \N_0$, we have \[\begin{aligned}
\max\mathsf L((A^{M}A_0^{\ell_{i_1}})^{t_1}(A^{M}A_0^{\ell_{i_2}})^{t_2})&=t_1\max\mathsf L(A^{M}A_0^{\ell_{i_1}})+t_2\max\mathsf L(A^{M}A_0^{\ell_{i_2}}) \quad \text{and}  \\
\min\mathsf L((A^{M}A_0^{\ell_{i_1}})^{t_1}(A^{M}A_0^{\ell_{i_2}})^{t_2})&=t_1\min\mathsf L(A^{M}A_0^{\ell_{i_1}})+t_2\min\mathsf L(A^{M}A_0^{\ell_{i_2}})\,.
\end{aligned}\]

\item[(b)]   If there is some $i\in [1,kN']$ such  that $\ell_i>\tau_M$, then
$\ell_j\ge \tau_M$ for all $j\in [1,kN']$.

\item[(c)]   If there is some $i\in [1,kN']$ such that $\ell_i>r_M$,
 then $\ell_j\ge r_M$ for all $j\in [1,kN']$.

\item[(d)]   If $k, \ell \in \N$ such that $\ell/k=r_M$, then $\ell_i=r_M$ for all $i\in [1, kN']$. In particular, $$\rho(\mathsf L((A^MA_0^{r_M})^{t_1}))=\rho(\mathsf L(A^MA_0^{r_M}))\qquad  \text{ for every $t_1\in \N$}\,.$$
\end{enumerate}

\end{enumerate}

\end{lemma}

\begin{proof}
We denote by $\mathsf Z (G_0) := \mathcal F ( \mathcal A (G_0))$ the factorization monoid of $\mathcal B (G_0)$ and by $\pi \colon  \mathsf Z (G_0) \rightarrow \mathcal B(G_0)$  the factorization homomorphism. If $z \in \mathsf Z (G_0)$, then $|z| = |z|_{\mathcal F ( \mathcal A (G_0))}$ denotes the length of $z$.

1. This follows from \cite[Lemma 5.4.1]{Ge-Sc-Zh17b}.

\medskip
2. Since $\supp(A_0)\subsetneq \supp(A)$, it follows that  $A^k\not| A_0^{\ell}$ for any $k, \ell\in \N$. We consider the monoid
\[
S=\{(x,y)\in \mathsf Z (G_0)\times \mathsf Z(G_0) \mid \pi(x)=\pi(y)=A^kA_0^{\ell}\in \mathcal B(G_0) \text{ for some $k \in \N_0$ and some $\ell\in \Z$}\}\,.
\]

Since $\mathsf Z (G_0) \times \mathsf Z (G_0)$ is finitely generated and the inclusion $S \hookrightarrow \mathsf Z (G_0)\times \mathsf Z (G_0)$ is a divisor homomorphism,  $S$ is  finitely generated by \cite[Proposition 2.7.5]{Ge-HK06a}.
Let  $\mathcal A(S)=\{(x_1, y_1), \ldots, (x_t, y_t)\}$ and $\pi(x_i)=\pi(y_i)=A^{k_i}A_0^{\ell_i}$ for every $i\in [1,t]$,
 where $t, k_1,\ldots, k_t\in \N_0$ and $\ell_1,\ldots, \ell_t\in \Z$.

Without loss of generality, we can assume that $\{k_1,\ldots, k_t\}=\{k_1,\ldots,k_{t'}\}$ for some $t'\in [1,t]$ and $k_i< k_j$ for any  $i,j\in [1,t']$ with $i<j$.  Since $(A_0,A_0) \in \mathcal A (S)$, it follows that $k_1=0$ and we define  $M=\prod_{i=2}^{t'}k_i$.

 Let $k \in \N$ and $\ell\in \N_0$. If $\ell=0$, then the claim is clear. Suppose that $\ell > 0$. We choose  $(x,y)\in \mathsf Z (G_0) \times \mathsf Z (G_0)$ with $\pi(x)=\pi(y)=A^{Mk}A_0^{\ell}$ such that $|x|=\min \mathsf L(A^{Mk}A_0^{\ell})$ and $|y|=\max \mathsf L(A^{Mk}A_0^{\ell})$.
We set $(x,y)=\prod_{i=1}^t(x_i,y_i)^{v_i}$, where $v_i\in \N_0$, and we observe that
\[
Mk = \sum_{i=1}^t v_ik_i \,.
\]
If $v_i\neq 0$, then $|x_i|=\min \mathsf L(A^{k_i}A_0^{\ell_i})$ and $|y_i|=\max \mathsf L(A^{k_i}A_0^{\ell_i})$.
Therefore we obtain that
\begin{align*}
\max\mathsf L(A^{Mk}A_0^{\ell})&=|y|=\sum_{i=1}^tv_i|y_i|=\sum_{i=1}^tv_i\max\mathsf L(A^{k_i}A_0^{\ell_i}) \quad \text{and} \\
\min\mathsf L(A^{Mk}A_0^{\ell})&=|x|=\sum_{i=1}^tv_i|x_i|=\sum_{i=1}^tv_i\min\mathsf L(A^{k_i}A_0^{\ell_i})\,.
\end{align*}
Let $I_i=\{j\in [1,t]\mid k_j=k_i\}$ for every $i\in [1,t']$.
Since
\[
k\prod_{i=2}^{t'}k_i  =Mk= \sum_{i=1}^tv_ik_i=\sum_{i=1}^{t'}(\sum_{j\in I_i}v_jk_j)=\sum_{i=1}^{t'}(\sum_{j\in I_i}v_jk_i) \overset{k_1=0!}{=} \sum_{i=2}^{t'} (\sum_{j\in I_i}v_j)k_i\,,
\]
 it follows by Lemma \ref{3.2} that there exist $x_i^{(j)}\in [0, v_i], i\in [1,t]\setminus I_1, j\in [1,k]$ such that $\sum_{i\in [1,t]\setminus I_1}x_i^{(j)}k_i=M$ for every $j\in [1,k]$ and  $\sum_{j=1}^kx_i^{(j)}=v_i$ for every $i\in [1,t]\setminus I_1$. Let
$\ell_1'=\sum_{i\in [1,t]\setminus I_1}x_i^{(1)}\ell_i+\sum_{i\in I_1}v_i\ell_i$
and $\ell_j'=\sum_{i\in [1,t]\setminus I_1}x_i^{(j)}\ell_i$ for every $j\in [2,k]$. Then  $\ell_1'+\ldots+\ell_k'=\ell$ and
\begin{align*}
\max\mathsf L(A^{Mk}A_0^{\ell})\ge \sum_{j=1}^k\max\mathsf L(A^{M}A_0^{\ell_j'}) &\ge \sum_{i=1}^tv_i\max\mathsf L(A^{k_i}A_0^{\ell_i})=\max\mathsf L(A^{Mk}A_0^{\ell})\,,\\
\min\mathsf L(A^{Mk}A_0^{\ell})\le \sum_{j=1}^k\min\mathsf L(A^{M}A_0^{\ell_j'}) &\le \sum_{i=1}^tv_i\max\mathsf L(A^{k_i}A_0^{\ell_i})=\min\mathsf L(A^{Mk}A_0^{\ell})\,.
\end{align*}

\medskip
3. We set $r_0= M|A|(\mathsf D(G_0)-1)$, choose $r \ge r_0$, and consider a factorization of $A^MA_0^r=V_1 \cdot \ldots \cdot V_{z}$, where $V_i\in \mathcal A(G_0)$ for every $i\in [1,z]$. Let $I\subset [1,z]$ be minimal such that $A^M\t \prod_{i\in I}V_i$. Then $|I|\le M|A|$. Since
\[
|(\prod_{i\in I}V_i)(A^M)^{-1}|\le |I|\mathsf D(G_0)-M|A|\le M|A|(\mathsf D(G_0)-1)=r_0 \,,
\]
we infer that $\prod_{i\in I}V_i\t A^MA_0^{r_0}$. Furthermore, the equation
 \[
 A^M A_0^{r_0}(\prod_{i \in I}V_i)^{-1} A_0^{r-r_0} = \prod_{i \notin I} V_i
 \]
implies that $A_0^{r-r_0}\t \prod_{i\notin I}V_i$. Since $\supp(\prod_{i\notin I}V_i)\subset \supp(A_0)$ is half-factorial, it follows that
\[
z - |I| \in \mathsf L ( \prod_{i \notin I} V_i) = \mathsf L(A_0^{r-r_0})+\mathsf L(\prod_{i\notin I}V_i(A_0^{r-r_0})^{-1})
\]
and
\[
\begin{aligned}
z=|I|+z-|I|\in & \mathsf L(\prod_{i\in I}V_i)+\mathsf L(A_0^{r-r_0})+\mathsf L(\prod_{i\notin I}V_i(A_0^{r-r_0})^{-1})\subset \mathsf L(A^MA_0^{r_0})+\mathsf L(A_0^{r-r_0}) \\
 & = (r-r_0) + \mathsf L(A^MA_0^{r_0}) \,.
\end{aligned}
\]
Since this holds true for every factorization of $A^M A_0^r$, it follows that
\[
\mathsf L (A^M A_0^r) \subset \mathsf L(A^MA_0^{r_0})+\mathsf L(A_0^{r-r_0}) \subset \mathsf L (A^M A_0^r)
\]
whence
\[
\max\mathsf L(A^MA_0^{r})=\max\mathsf L(A^MA_0^{r_0})+r-r_0 \quad \text{and} \quad
\min\mathsf L(A^MA_0^{r})=\min\mathsf L(A^MA_0^{r_0})+r-r_0\,.
\]
Therefore, $r_M$ and $\tau_M$ are well-defined and, by definition, we have $\tau_M < r_M$.

\medskip
4. Let $k, \ell \in \N_0$ by given and note that for $k=0$ or $\ell = 0$, the statement is clear. Suppose that $k, \ell \in \N$.  Then \cite[Theorem 3.8.1.2]{Ge-HK06a} implies that
there exists $N = N(k, \ell) \in \N$ such that for every $t\in \N$, we have
\[
\max \mathsf L( (A^{MkN}A_0^{\ell N})^{t} )=t\max \mathsf L( A^{MkN}A_0^{\ell N} ) \ \text{ and }\ \min \mathsf L( (A^{MkN}A_0^{\ell N})^{t} )=t\min\mathsf L( A^{MkN}A_0^{\ell N} )\,.
\]
Let $N'$ be a multiple of   $N$. It follows from 2. that there  exist $\ell_1,\ldots, \ell_{kN'}\in \Z$ with $\ell_1+\ldots+\ell_{kN'}=\ell N'$ such that all $A^MA_0^{\ell_i} \in \mathcal B (G_0)$ and
\begin{align*}
\max\mathsf L(A^{MkN'}A_0^{\ell N'})=\sum_{i=1}^{kN'}\max\mathsf L(A^MA_0^{\ell_i})\ \  \text{ and }\ \
\min\mathsf L(A^{MkN'}A_0^{\ell N'})=\sum_{i=1}^{kN'}\min\mathsf L(A^MA_0^{\ell_i})\,.
\end{align*}

\smallskip
Now let $\ell_1,\ldots, \ell_{kN'}\in \Z$ satisfy all these properties.

\medskip
(a)  Let $i_1, i_2 \in [1, kN']$ be distinct, say $i_1=1$ and $i_2=2$, and let $t_1, t_2 \in \N_0$, say  $t=\max\{t_1,t_2\}$.
Obviously, we have
\begin{equation} \label{crucial6}
\begin{aligned}
(A^{MkN'}A_0^{\ell N'})^t & = \prod_{i=1}^{kN'} ( A^M A_0^{\ell_i})^t \\
 & = \left( (A^MA_0^{\ell_1})^{t-t_1} \right) \left( (A^MA_0^{\ell_2})^{t-t_2} \right) \left( (A^M A_0^{\ell_1})^{t_1}(A^MA_0^{\ell_2})^{t_2} \right) \left( \prod_{i=3}^{kN'} ( A^M A_0^{\ell_i})^t\right)
\end{aligned}
\end{equation}
and hence we get
$$\begin{aligned}
t\sum_{i=1}^{kN'}\max\mathsf L(A^MA_0^{\ell_i})&=\max\mathsf L((A^{MkN'}A_0^{\ell N'})^t)\\
& \overset{\eqref{crucial6}}{\ge}
(t-t_1)\max\mathsf L(A^MA_0^{\ell_1})+
(t-t_2)\max\mathsf L(A^MA_0^{\ell_2}) \\ &\qquad\qquad\qquad \qquad + \max \mathsf L\big((A^MA_0^{\ell_1})^{t_1}(A^MA_0^{\ell_2})^{t_2}\big) +t\sum_{i=3}^{kN'}\max\mathsf L(A^MA_0^{\ell_i})\\
&\ge t\sum_{i=1}^{kN'}\max\mathsf L(A^MA_0^{\ell_i})\,.
\end{aligned} $$
It follows that $$\max \mathsf L\big((A^MA_0^{\ell_1})^{t_1}(A^MA_0^{\ell_2})^{t_2}\big)=t_1\max\mathsf L(A^MA_0^{\ell_1})+t_2\max\mathsf L(A^MA_0^{\ell_2})\,.$$
By the similar argument, we can obtain $$\min \mathsf L\big((A^MA_0^{\ell_1})^{t_1}(A^MA_0^{\ell_2})^{t_2}\big)=t_1\min\mathsf L(A^MA_0^{\ell_1})+t_2\min\mathsf L(A^MA_0^{\ell_2})\,.$$

\medskip
(b) Suppose there exist some $i\in [1,kN']$ such that $\ell_i>\tau_M$, say $i=1$. Thus $\rho(\mathsf L(A^MA_0^{\ell_1}))<\rho(\mathsf L(A^M))$ by the definition of $\tau_M$.
Assume to the contrary that there exists some $i\in [1,kN']$ such that  $\ell_i <\tau_M$, say $i=2$.
Then $$\begin{aligned}
\rho(\mathsf L(A^M))&=\rho(\mathsf L(A^MA_0^{\tau_M}))=\rho(\mathsf L((A^MA_0^{\tau_M})^{\ell_1-\ell_2}))\\
&=\frac{\max\mathsf L(A^{M(\ell_1-\ell_2)}A_0^{\tau_M(\ell_1-\ell_2)})}{\min\mathsf L(A^{M(\ell_1-\ell_2)}A_0^{\tau_M(\ell_1-\ell_2)})}\\
&=\frac{\max\mathsf L\big((A^MA_0^{\ell_1})^{\tau_M-\ell_2}(A^MA_0^{\ell_2})^{\ell_1-\tau_M}\big)}{\min\mathsf L\big((A^MA_0^{\ell_1})^{\tau_M-\ell_2}(A^MA_0^{\ell_2})^{\ell_1-\tau_M}\big)}\\
& \overset{(a)}{=} \frac{(\ell_1-\tau_M)\max\mathsf L(A^MA_0^{\ell_2})+(\tau_M-\ell_2)\max\mathsf L(A^MA_0^{\ell_1})}{(\ell_1-\tau_M)\min\mathsf L(A^MA_0^{\ell_2})+(\tau_M-\ell_2)\min\mathsf L(A^MA_0^{\ell_1})}\\
&<\rho(\mathsf L(A^M))\,,
\end{aligned}$$
a contradiction.

\medskip
(c). Suppose there exist some $i\in [1,kN']$ such that $\ell_i>r_M$, say $i=1$.
Assume to the contrary that there exists some $i\in [1,kN']$ such that  $\ell_i<r_M$,  say $i=2$.
Then
\begin{align*}
\max\mathsf L( & (A^MA_0^{r_M})^{\ell_1-\ell_2})=\max\mathsf L((A^MA_0^{\ell_1})^{r_M-\ell_2}(A^MA_0^{\ell_2})^{\ell_1-r_M}) \\
& \overset{(a)}{=}(r_M-\ell_2)\max\mathsf L(A^MA_0^{\ell_1})+(\ell_1-r_M)\max\mathsf L(A^MA_0^{\ell_2})\\
& \overset{\text{def of} \ r_M }{=} (r_M-\ell_2)\max\mathsf L(A^MA_0^{r_M})+(r_M-\ell_2)(\ell_1-r_M)+(\ell_1-r_M)\max\mathsf L(A^MA_0^{\ell_2})\\
& = (r_M-\ell_2)\max\mathsf L(A^MA_0^{r_M})+(\ell_1-r_M)\max \mathsf L ( A_0^{r_M-\ell_2})+(\ell_1-r_M)\max\mathsf L(A^MA_0^{\ell_2})\\
& \le (r_M-\ell_2)\max\mathsf L(A^MA_0^{r_M})+(\ell_1-r_M)\max\mathsf L(A^MA_0^{r_M})\\
& =(\ell_1-\ell_2)\max\mathsf L(A^MA_0^{r_M})\\
& \le \max\mathsf L((A^MA_0^{r_M})^{\ell_1-\ell_2})\,,
\end{align*}
which implies  that $r_M-\ell_2+\max L(A^MA_0^{\ell_2})=\max \mathsf L(A^MA_0^{r_M})$. Since
$$\max \mathsf L(A^MA_0^{r_M})\ge \max\mathsf L(A^MA_0^{r_M-1})+1\ge \ldots \ge \max \mathsf L(A^MA_0^{\ell_2})+r_M-\ell_2=\max \mathsf L(A^MA_0^{r_M})\,,$$
we obtain for every $r\ge \ell_2$,
\begin{align*}
\max\mathsf L(A^MA_0^{r})=\max\mathsf L(A^MA_0^{\ell_2})+r-\ell_2\,.
\end{align*}
Similarly, we can prove that $\min\mathsf L(A^MA_0^{r})=\min\mathsf L(A^MA_0^{\ell_2})+r-\ell_2$ for every $r\ge \ell_2$. Since $\ell_2 < r_M$, this is
a contradiction to the minimality of $r_M$.

\medskip
(d) Suppose that $\ell/k=r_M$. Assume to the contrary that there is an $i \in [1, kN']$ with $\ell_i > r_M$, then (c) implies that $\ell_j \ge  r_M$ for all $j \in [1, kN']$ whence
\[
r_M = \frac{\ell N'}{k N'} = \frac{\sum_{j=1}^{kN'} \ell_j}{kN'} > r_M \,,
\]
a contradiction. Thus $\ell_i \le r_M$ for all $i \in [1, kN']$ and  $r_M = \frac{\sum_{j=1}^{kN'} \ell_j}{kN'}$ shows that $\ell_i = r_M$ for all $i \in [1, kN']$.
Now the  "in particular" statement  follows immediately from (a) (with $t_2=0$).
\end{proof}

\begin{lemma} \label{3.4}
Let $G$ be an  abelian group and  $G_0 \subset G$ a finite subset.  Then $\{ \rho (L) \mid L \in \mathcal L (G_0) \} = \{ x \in \Q \mid 1 \le x \le \rho (G_0) \}$.
\end{lemma}

\begin{proof}
Since $G_0$ is finite, $\mathcal B (G_0)$ has accepted elasticity (\cite[Theorem 3.1.4]{Ge-HK06a}) whence  there is
 $A\in \mathcal B(G_0)$ with $\rho(A)=\rho(G_0)$. Since the assertion is obvious when $\rho(A)=1$, we assume that $\rho(A)>1$.
 It follows by Lemma \ref{3.3}.1 that there exists an atom $A_0\in \mathcal A(G_0)$ such that $\supp(A_0)\subsetneq \supp(A)$ is half-factorial.
Our strategy is to define rational functions $f(k/\ell) = \rho \big( \mathsf L (A^{ak}A_0^{b\ell}) \big)$, for suitable $a, b \in \N$ and all $k, \ell \in \N$, which are surjective on the rational interval between $1$ and $\rho (G_0)$. This would be very simple if $A_0$ would be a prime but, in general, $\mathcal B (G_0)$ need not contain prime elements. We proceed in two steps. In the first step we define partitions of the rational intervals between $\tau_M$ and $r_M$ and between $1$ and $\rho (G_0)$. In the second step we define surjective rational functions between the constructed subsets.

\smallskip
\noindent
{\bf Step 1.} Let $M$, $\tau_M$, and $r_M$ be defined as  in Lemma \ref{3.3} whence, in particular, we have
\[
\rho(\mathsf L(A^MA_0^{\tau_M}))=\rho(\mathsf L(A^M)) = \rho ( \mathsf L (A)) = \rho (G_0) \,.
\]
We  define a subset $I_M \subset [\tau_M, r_M]$ by saying that an element $t\in [\tau_M, r_M]$ lies in $I_M$ if and only if for every $n \in \N$, we have
\begin{itemize}
\item $
\max \mathsf L((A^MA_0^t)^{n})=n\max\mathsf L(A^MA_0^{t})\ \text{ and }\
\min \mathsf L((A^MA_0^t)^{n})=n \min\mathsf L(A^MA_0^{t})$.
\end{itemize}
By Lemma \ref{3.3}.4.d, we have $r_M\in I_M$, and since
for every $n\in \N$, $\rho(\mathsf L((A^MA_0^{\tau_M})^{n}))=\rho(\mathsf L(A^MA_0^{\tau_M}))$, it follows that  that $\tau_M\in I_M$.

We set $I_M=\{t_0,   \ldots, t_s\}$ with $s \in \N$ and $\tau_M=t_0´<\ldots<t_s=r_M$.
For any $i,j\in [0, s]$ with $i<j$, we have that
\begin{align*}
\rho(\mathsf L(A^MA_0^{t_i}))&=\rho(\mathsf L((A^MA_0^{t_i})^{t_j-t_0}))=\rho(\mathsf L((A^MA_0^{t_0})^{t_j-t_i}(A^MA_0^{t_j})^{t_i-t_0}))\\
&\ge\frac{(t_j-t_i)\max\mathsf L(A^MA_0^{t_0})+(t_i-t_0)\max\mathsf L(A^MA_0^{t_j})}{(t_j-t_i)\min\mathsf L(A^MA_0^{t_0})+(t_i-t_0)\min\mathsf L(A^MA_0^{t_j})}\\
&\ge \min \{\rho(\mathsf L(A^MA_0^{t_0})), \rho(\mathsf L(A^MA_0^{t_j}))\} \\
&=\rho(\mathsf L(A^MA_0^{t_j}))\,.
\end{align*}
Thus it suffices to prove that
\begin{align*}
\{ q \in \Q \mid \rho(\mathsf L(A^MA_0^{t_{j+1}})) < q < \rho (\mathsf L(A^MA_0^{t_j})) \}&\subset \{ \rho (L) \mid L \in \mathcal L (G_0) \} \text{ for every $j\in [0, s-1]$ and}\\
\{ q \in \Q \mid 1 < q < \rho (\mathsf L(A^MA_0^{r_M})) \}&\subset \{ \rho (L) \mid L \in \mathcal L (G_0) \}\,.
\end{align*}

\smallskip
\noindent
{\bf Step 2.}
Let  $k,\ell\in \N$ with $\ell/k>\tau_M$. Then Lemma \ref{3.3}.4 implies that  there exists $N = N(k, \ell) \in \N$ such that for every $t\in \N$,
\[
\begin{aligned}
\max\mathsf L((A^{MkN}A_0^{\ell N})^t)=t\max\mathsf L(A^{MkN}A_0^{\ell N})\,, \quad
\min\mathsf L((A^{MkN}A_0^{\ell N})^t)=t\min\mathsf L(A^{MkN}A_0^{\ell N})\,,
\end{aligned}
\]
and  there are $\ell_1, \ldots, \ell_{kNt}\in \Z$ such that
\begin{equation} \label{crucial3}
\ell_1+\ldots+\ell_{kNt}=\ell Nt
\end{equation}
and
\begin{equation} \label{crucial1}
\max\mathsf L(A^{MkNt}A_0^{\ell Nt})=\sum_{i=1}^{kNt}\max\mathsf L(A^MA_0^{\ell_i}) \  \text{ and }\
\min\mathsf L(A^{MkNt}A_0^{\ell Nt})=\sum_{i=1}^{kNt}\min\mathsf L(A^MA_0^{\ell_i})\,.
\end{equation}

Let $t \in \N$.
Since $\ell/k>\tau_M$, there exists some $i\in [1,kNt]$ such that $\ell_i>\tau_M$. Then Lemma \ref{3.3}.4.b implies that $\ell_j\ge \tau_M$ for all $j\in [1,kNt]$. Now we choose $\ell_1, \ldots, \ell_{kNt}\in \N_{\ge \tau_M}$ with the above  properties such that
\[
C_t= \Big| I_M\cap \big[\min\{\ell_i\mid i\in[1,kNt]\}, \max\{\ell_i\mid i\in[1,kNt]\} \big] \Big|
\]
is minimal. Let $t^* \in \N$ such that $C_{t^*}=\min \{C_t\mid t\in \N\}$ and $\ell_1,\ldots, \ell_{kNt^*}$ are the associated  integers in $\N_{\ge \tau_M}$.
We distinguish two cases.

\smallskip
\noindent CASE 1: \, $\ell/k>r_M$.

Then there exists some $i\in [1,kNt^*]$ such that $\ell_i>r_M$. Lemma \ref{3.3}.4.c implies that $\ell_j\ge r_M$ for all $j\in[1,kNt^*]$. Therefore, by \eqref{crucial1} and by the definition of $r_M$,
\[
\begin{aligned}
\rho(\mathsf L(A^{MkNt^*}A_0^{\ell N t^*})) &
= \frac{\sum_{i=1}^{kNt^*} \max \mathsf L (A^M A_0^{l_i})}{\sum_{i=1}^{kNt^*} \min \mathsf L (A^M A_0^{l_i})}
=\frac{\sum_{i=1}^{kNt^*}(\ell_i-r_M)+kNt^*\max\mathsf L(A^MA_0^{r_M})}{\sum_{i=1}^{kNt^*}(\ell_i-r_M)+kNt^*\min\mathsf L(A^MA_0^{r_M})}
\\ & = \frac{\frac{\ell}{k}-r_M+\max\mathsf L(A^MA_0^{r_M})}{\frac{\ell}{k}-r_M+\min\mathsf L(A^MA_0^{r_M})}\,,
\end{aligned}
\]
which implies that
\[
\{ \rho(\mathsf L(A^{MkNt^*}A_0^{\ell N t^*})) \mid k, \ell \in \N \ \text{with} \ \ell/k>r_M \} = \{q \in \Q\mid 1<q<\rho(\mathsf L(A^MA_0^{r_M}))\} \,.
\]

\smallskip
\noindent
CASE 2: \,  $\tau_M<\ell/k\le r_M$.

Since $\ell_1,\ldots, \ell_{kNt^*} \in \N_{\ge \tau_M}$, Lemma \ref{3.3}.4.c. implies that
$\ell_i\in [\tau_M, r_M]$ for all $i\in [1,kNt^*]$.  It follows by Lemma \ref{3.3}.4.a (with $t_2=0$) that  $\ell_i\in I_M$ for all $i\in [1,kNt^*]$. After renumbering if necessary we suppose that
\[
\ell_1=\max\{\ell_i\mid i\in [1,kNt^*]\} \quad \text{and} \quad \ell_2=\min\{\ell_i\mid i\in [1,kNt^*]\} \,.
\]
Assume to the contrary that  $C_{t^*}\ge 3$. Then there is a $y\in I_m$ such that $\ell_2<y<\ell_1$. We set $S = \ell_1 \cdot \ldots \cdot \ell_{k Nt^*} \in \mathcal F ( \N)$ and distinguish two cases.

\smallskip
\noindent
CASE 2.1: \,  $\mathsf v_{\ell_1}(S)(\ell_1-y)\le \mathsf v_{\ell_2}(S)(y-\ell_2)$.

We will define integers $\ell_1', \ldots, \ell_{kNt^* (y-\ell_2)}' \in \N_{\ge \tau_M}$ satisfying \eqref{crucial3} and \eqref{crucial1} such that $C_{t^*(y-\ell_2)} < C_{t^*}$, which will be a contradiction to the minimality of $C_{t^*}$.

We define
\[
\ell_1' \cdot \ldots \cdot \ell_{kNt^*(y-\ell_2)}' := \ell_2^{(y-\ell_2)\mathsf v_{\ell_2}(S)-\mathsf v_{\ell_1}(S)(\ell_1-y)}y^{(\ell_1-\ell_2)\mathsf v_{\ell_1}(S)}\prod_{i=3}^{kNt^*}\ell_i^{y-\ell_2}  \in \mathcal F (\N_{\ge \tau_M}) \,.
\]
Clearly, we have $\sum_{i=1}^{kNt^*(y-\ell_2)}\ell_i'=\ell N t^*(y-\ell_2)$ and hence \eqref{crucial3} holds.
Furthermore, since $y \in I_M$ and by Lemma \ref{3.3}.4.a we get
\begin{equation} \label{crucial4}
\begin{aligned}
(\ell_1-\ell_2)\max\mathsf L(A^MA_0^y) & =
\max\mathsf L((A^MA_0^y)^{\ell_1-\ell_2})=\\
\qquad \max\mathsf L((A^MA_0^{\ell_1})^{y-\ell_2}(A^MA_0^{\ell_2})^{\ell_1-y}) & =
(y-\ell_2)\max\mathsf L(A^MA_0^{\ell_1})+(\ell_1-y)\max\mathsf L(A^MA_0^{\ell_2})
\end{aligned}
\end{equation}
and
\begin{equation}\label{crucial5}
\begin{aligned}
(\ell_1-\ell_2)\min\mathsf L(A^MA_0^y) & =
\min\mathsf L((A^MA_0^y)^{\ell_1-\ell_2})=\\
\qquad \min\mathsf L((A^MA_0^{\ell_1})^{y-\ell_2}(A^MA_0^{\ell_2})^{\ell_1-y}) & =
(y-\ell_2)\min\mathsf L(A^MA_0^{\ell_1})+(\ell_1-y)\min\mathsf L(A^MA_0^{\ell_2})\,.
\end{aligned}
\end{equation}
Putting all together we obtain that, by using Lemma \ref{3.3}.4 (with $t=y-\ell_2$) and  \eqref{crucial1} (with $t^*$) for $(*)$,
\begin{align*}
&\max\mathsf L(A^{MkNt^*(y-\ell_2)}A_0^{\ell N t^*(y-\ell_2)}) \overset{(*)}{=} (y-\ell_2)\sum_{i=1}^{kNt^*}\max\mathsf L(A^MA_0^{\ell_i}) \overset{\eqref{crucial4}}{=}\sum_{i=1}^{kNt^*(y-\ell_2)}\max\mathsf L(A^MA_0^{\ell_i'})\\
&\min\mathsf L(A^{MkNt^*(y-\ell_2)}A_0^{\ell N t^*(y-\ell_2)} ) \overset{(*)}{=} (y-\ell_2)\sum_{i=1}^{kNt^*}\min\mathsf L(A^MA_0^{\ell_i}) \overset{\eqref{crucial5}}{=} \sum_{i=1}^{kNt^*(y-\ell_2)}\min\mathsf L(A^MA_0^{\ell_i'})\,.
\end{align*}
Therefore, \eqref{crucial1} holds and since $C_{t^*(y-\ell_2)}<C_{t^*}$, we get a contradiction to the minimality of $C_{t^*}$.

\smallskip
\noindent
CASE 2.2: \, $\mathsf v_{\ell_1}(S)(\ell_1-y)\ge \mathsf v_{\ell_2}(S)(y-\ell_2)$.

The proof runs along the same lines as the proof of CASE 2.1.

\smallskip
Therefore, we obtain that   $C_{t^*}\le 2$. Let $j\in [0, s-1]$ such that
$t_j<\ell/k<t_{j+1}$. Note that
\[
\ell_2 = \min \{\ell_i \mid i \in [1, kNt^*]\} \le \frac{\sum_{i=1}^{kNt^*}\ell_i}{kNt^*} = \frac{\ell}{k} \le \max \{\ell_i \mid i \in [1, kNt^*]\} = \ell_1 \,.
\]
If $\ell_1=\ell_2$, then $\ell/k = \ell_1$, a contradiction to $t_j<\ell/k<t_{j+1}$. Thus we get that $\ell_2 < \ell/k < \ell_1$. Since $t_j$ is the maximal element of $I_M$ that is smaller than $\ell/k$ and $t_{j+1}$ is the minimal element of $I_M$ that is larger than  $\ell/k$, we obtain that $\ell_2 \le t_j < t_{j+1} \le \ell_1$ whence $\{t_j, t_{j+1} \} \subset I_M \cap [\ell_2, \ell_1]$. Since $C_{t^*} \le 2$ and $\{\ell_i \mid i \in [1, kNt^*] \} \subset I_M \cap [\ell_2, \ell_1]$, it follows that
\begin{equation} \label{crucial2}
\{\ell_i\mid i\in [1,kNt^*]\}= I_M \cap [\ell_2, \ell_1] = \{t_j, t_{j+1}\} \,.
\end{equation}
Then
\[
\begin{aligned}
\rho(A^{MkNt^*}A_0^{\ell Nt^*}) & = \frac{\max \mathsf L ( A^{MkNt^*}A_0^{\ell Nt^*} )}{\min \mathsf L ( A^{MkNt^*}A_0^{\ell Nt^*} )} \\
& \overset{\eqref{crucial1}}{=} \frac{\sum_{i=1}^{kNt^*} \max \mathsf L ( A^M A_0^{\ell_i})}{\sum_{i=1}^{kNt^*} \min \mathsf L ( A^M A_0^{\ell_i})}  \\
& =\frac{x_1\max\mathsf L(A^MA_0^{t_{j}})+x_2\max\mathsf L(A^MA_0^{t_{j+1}})}{x_1\min\mathsf L(A^MA_0^{t_{j}})+x_2\min\mathsf L(A^MA_0^{t_{j+1}})} \,,
\end{aligned}
\]
where by \eqref{crucial2}
\[
x_1 = |\{i \in [1, kNt^*] \mid \ell_i = t_j\}| \quad \text{and} \quad x_2 = |\{i \in [1, kNt^*] \mid \ell_i = t_{j+1}\}| \,.
\]
Comparing exponents of $A$ and $A_0$ we obtain the equations
\[
x_1t_j + x_2t_{j+1} = \ell N t^*  \quad \text{and} \quad x_1M + x_2M = Mk N t^*
\]
whence
\[
(x_1,x_2)= \left(\frac{(kt_{j+1}-\ell)Nt^*}{t_{j+1}-t_j}, \frac{(\ell-kt_j)Nt^*}{t_{j+1}-t_j} \right) \,.
\]
Plugging in this expression for $(x_1, x_2)$ we obtain that
\[
\begin{aligned}
\rho(A^{MkNt^*}A_0^{\ell Nt^*}) & =
 \frac{ (k t_{j+1} - \ell) \max\mathsf L(A^MA_0^{t_{j}}) + (\ell - kt_j)\max\mathsf L(A^MA_0^{t_{j+1}})}{(k t_{j+1} - \ell) \min \mathsf L(A^MA_0^{t_{j}}) + (\ell - kt_j)\min \mathsf L(A^MA_0^{t_{j+1}})} \\
& =  \frac{ ( t_{j+1} - \ell/k) \max\mathsf L(A^MA_0^{t_{j}}) + (\ell/k - t_j)\max\mathsf L(A^MA_0^{t_{j+1}})}{(t_{j+1} - \ell/k) \min \mathsf L(A^MA_0^{t_{j}}) + (\ell/k - t_j)\min \mathsf L(A^MA_0^{t_{j+1}})} \,.
\end{aligned}
\]
Thus, if $\ell/k$ varies between $t_j$ and $t_{j+1}$, then $\rho (A^{MkNt^*}A_0^{\ell Nt^*})$ varies between $\rho(\mathsf L(A^MA_0^{t_{j+1}}))$ and $\rho(\mathsf L(A^MA_0^{t_{j}}))$.
\[
\{q \in \Q \mid \rho(\mathsf L(A^MA_0^{t_{j+1}}))< q < \rho(\mathsf L(A^MA_0^{t_{j}}))\}\subset \{\rho(L)\mid L\in \mathcal L(G_0)\}\,. \qedhere
\]
\end{proof}

\noindent
\begin{proof}[{\bf Proof of Theorem \ref{3.1}. }]
Suppose that $H$ is a transfer Krull monoid over the subset $G_0$ of an abelian group $G$. By \eqref{transfer}, we have $\mathcal L (H)= \mathcal L (G_0)$ whence it is sufficient to prove the assertion for $\mathcal B (G_0)$.
Let $q \in \Q$ with $1\le q <\rho(G_0)$. Since
\[
\rho(G_0) = \sup \big\{ \rho \big( \mathsf L (B) \big) \mid B \in \mathcal B (G_0) \big\} \,,
\]
there exists a $B \in \mathcal B (G_0)$ with  $\rho(\mathsf L(B))>q$. Then $G_1=\supp(B)$ is finite and $\rho(G_1)\ge \rho(\mathsf L(B))>q$. Thus Lemma \ref{3.4} implies that
\[
q \in  \{\rho(L) \mid L \in \mathcal L (G_1) \} \subset  \{\rho(L) \mid L \in \mathcal L (G_0) \}  \,.  \qedhere
\]
\end{proof}

\medskip
\section{The set of catenary degrees and related sets} \label{4}
\medskip

In Subsection \ref{sec-4.1} we introduce catenary degrees and related  invariants. The main results are formulated in Subsection \ref{sec-4.2} where also crucial examples are discussed. The proof of the main results are given in Subsection \ref{sec-4.3}.

\subsection{Factorizations and  catenary degrees.} \label{sec-4.1}
Transfer Krull monoids need not be commutative but they allow to shift the study of catenary degrees to the commutative setting (see Lemma \ref{4.1}). Thus we only briefly recall the concepts of factorizations, distance functions, and catenary degrees in general  monoids and refer to  the exposition by Baeth and Smertnig \cite[Section 4]{Ba-Sm15} for details. We provide additional explanations in the commutative case because this is the setting we are working in.

Let $H$ be a BF-monoid. We denote by $\mathsf Z^* (H)$ the monoid of rigid factorizations, by $\pi \colon \mathsf Z^* (H) \to H$ the factorization homomorphism, and by
\[
\mathsf d \colon \{(z,z') \in \mathsf Z^* (H) \time \mathsf Z^* (H) \mid \pi (z)=\pi (z') \} \to \mathbb N_0
\]
a  distance function.
For an element $a \in H$, the {\it catenary degree} $\mathsf c_{\mathsf d} (a) \in \mathbb N_0 \cup \{\infty\}$ (of $a$ with respect to the distance function $\mathsf d$) is the minimal $N \in \mathbb N_0 \cup \{\infty\}$ such that for any two factorizations $z, z'$ of $a$ there are factorizations $z=z_0, z_1, \ldots, z_n=z'$ of $a$ such that $\mathsf d (z_{i-1},z_i) \le N$ for all $i \in [1,n]$. Since $\mathsf c_{\mathsf d} (a) \le \max \mathsf L (a)$, the catenary degrees of all elements are finite. Then
\[
\Ca_{\mathsf d} (H) = \{\mathsf c_{\mathsf d} (a) \mid a \in H \ \text{with} \ \mathsf c_{\mathsf d} (a) \ge 2  \} \subset \N_0
\]
denotes the {\it set of  catenary degrees} and its supremum $\mathsf c_{\mathsf d} (H) = \sup \Ca_{\mathsf d} (H)$ is called the {\it catenary degree} of $H$ (we use the convention that $\sup \emptyset = 0$). Note that $\mathsf c_{\mathsf d} (a)=0$ for every atom $a \in \mathcal A (H)$. We impose that restriction $\mathsf c_{\mathsf d} (a) \ge 2$ in order to simplify the statements of our results   (see  \eqref{distance-catenary} and the discussion proceeding Proposition \ref{4.a}).
If $\theta \colon H \to B$ is a transfer homomorphism, then $\mathsf c_{\mathsf d} (H, \theta)$ denotes the catenary degree in the fibres.

Closely related to the set of catenary degrees is the
{\it set of minimal relations}  $\mathcal R_{\mathsf d} (H)$ which is defined as   the set of all $d \in \N_{\ge 2}$ with the following property:
\begin{itemize}
\item[] There are an element $a \in H$ and two distinct factorizations $z, z'$ of $a$ with $\mathsf d (z,z') = d$ such that  there is no $(d-1)$-chain of factorizations concatenating $z$ and $z'$.
\end{itemize}
The {\it set of distances} of $H$ (also called the delta set of $H$), defined as
\[
\Delta (H) = \bigcup_{L \in \mathcal L (H)} \Delta (L) \subset \N \,,
\]
is one of the oldest invariants in factorization theory. It is easy to check that $\min \Delta (H)=\gcd \Delta (H)$ (\cite[Proposition 3]{Ge-HK06a}).
The set
\[
\begin{aligned}
\daleth^* (H)&  = \{ \min (\mathsf L (uv) \setminus \{2\}) \mid u,v \in \mathcal A (H), |\mathsf L (uv)|>1\}  \\
 & = \{ \min (L  \setminus \{2\}) \mid 2 \in L \in \mathcal L (H), |L|>1 \} \subset \N_{\ge 3} \,,
\end{aligned}
\]
is a technical tool to study sets of catenary degrees. We have
\begin{equation} \label{inclusion1}
\daleth^* (H) \subset 2+\Delta (H) \subset \N_{\ge 3} \,,
\end{equation}
\begin{equation} \label{inclusion2}
\Ca_{\mathsf d} (H) \subset \mathcal R_{\mathsf d} (H), \ \sup \Ca_{\mathsf d} (H)  = \sup \mathcal R_{\mathsf d} (H) \,, \ \text{and} \quad
\end{equation}
\begin{equation} \label{inclusion3}
\text{if $H$ is commutative, then} \quad \daleth^* (H) \subset \mathcal R (H) \subset \N_{\ge 2} \,.
\end{equation}
The inclusion \eqref{inclusion1} follows from the definitions. We refer to  \cite[Lemma 2.1]{Fa-Ge17a} for the (very simple) proof of \eqref{inclusion2} and \eqref{inclusion3} in the commutative setting. The proof of \eqref{inclusion2} in the general setting runs along the same lines, but \eqref{inclusion3} need not hold in the general setting (see \cite[Remark 4.3]{Ba-Sm15}). We freely use the relations \eqref{inclusion1}, \eqref{inclusion2}, and \eqref{inclusion3}.

\smallskip
Now we shed some extra light on these invariants in the commutative setting. Suppose that $H$ is a commutative BF-monoid. The free abelian monoid $\mathsf Z (H) = \mathcal F (\mathcal A (H_{\red}))$ with basis $\mathcal A (H_{\red})$ is the factorization monoid of $H$, $\pi \colon \mathsf Z (H) \to H_{\red}$ is the factorization homomorphism, and $\mathsf Z (a) = \pi^{-1} (a)$ is the set of factorizations of $a$ for every $a \in H$.
The permutable distance function $\mathsf d_p$ coincides with the usual distance $\mathsf d$. Thus, if $u_1, \ldots, u_k, v_1, \ldots, v_{\ell}, w_1, \ldots, w_m$ are atoms of $H_{\red}$ such that $v_i \ne w_j$ for all $i \in [1, \ell]$ and all $j \in [1,m]$, then for the factorizations
\[
z = u_1 \cdot \ldots \cdot u_k v_1\cdot \ldots \cdot v_{\ell} \in \mathsf Z (H) \quad \text{and} \quad  z' = u_1 \cdot \ldots \cdot u_k w_1\cdot \ldots \cdot w_{m} \in \mathsf Z (H)
\]
we have
\begin{equation} \label{dist-commutative}
\mathsf d ( z, z') = \max \{\ell, m\} \in \N_0 \,.
\end{equation}
We briefly set
\[
\mathsf c (a) = \mathsf c_{\mathsf d} (a), \ \mathsf c (H) = \mathsf c_{\mathsf d} (H), \ \Ca (H) = \Ca_{\mathsf d} (H), \ \text{and} \ \mathcal R (H) = \mathcal R_{\mathsf d} (H) \,.
\]
Note that $\mathsf c (a)=0$ if and only if $|\mathsf Z (a)|=1$. If  $|\mathsf Z (a)| \ge 2$, then \eqref{dist-commutative} shows  (details can be found in \cite[Lemma 1.6.2]{Ge-HK06a}) that
\begin{equation} \label{distance-catenary}
2 + \max \Delta (\mathsf L (a)) \le \mathsf c (a) \quad \text{whence} \quad 2 + \max \Delta (H) \le \mathsf c (H) \,.
\end{equation}
In the non-commutative setting  $\mathsf c_{\mathsf d}(a)=1$ can happen (\cite[Remark 4.3]{Ba-Sm15}), but for distances of interest in HNP rings we have $\mathsf c_{\mathsf d} (a) \in \{0\} \cup \N_{\ge 2}$ (\cite[Proposition 4.12]{Sm18a}).
For a subset $G_0$ of an abelian group, we  set (as usual)
\[
\Ca (G_0) := \Ca \big( \mathcal B (G_0) \big), \ \Delta (G_0) := \Delta \big( \mathcal B (G_0) \big),
\]
and similarly for all the remaining invariants.

\medskip
\subsection{Main Results and Examples.} \label{sec-4.2}
We formulate the main results of this section (Theorem \ref{4.b} and Proposition \ref{4.a}) which state that - under the given assumptions - all the sets introduced in Subsection \ref{sec-4.1} are intervals. Recall that, by \eqref{Davenport2}, we have $\mathsf D (G)=3 \ \text{ if and only if} \  G \cong C_3 \ \text{ or} \  G \cong C_2 \oplus C_2$.

\begin{theorem} \label{4.b}
Let $H$ be a transfer Krull monoid having a transfer homomorphism $\theta \colon H \to \mathcal B (G)$  to an abelian group $G$ and a distance $\mathsf d$ with $\mathsf c_{\mathsf d} (H, \theta) \le 2$.
\begin{enumerate}
\item If $\mathsf D (G)\le 2$, then $\Delta (H)=\daleth^* (H)=\emptyset$ and if $\mathsf D (G)=3$, then $\Delta (H)=\{1\}$ and $\daleth^* (H)= \{3\}$.

\item Suppose that $G$ is finite with $\mathsf D (G) \ge 4$. Then $\Delta (H)$, $\daleth^* (H)$, and $\Ca_{\mathsf d} (H) = \mathcal R_{\mathsf d} (H) = [2, \mathsf c_{\mathsf d} (H)]$ are intervals. If $\mathsf D (G)=\mathsf D^* (G)$, then
      \[
      \daleth^* (H) \cup \{2\} = \big(2+\Delta (H)\big) \cup \{2\}   =  \Ca_{\mathsf d} (H) = \mathcal R_{\mathsf d} (H) = [2, \mathsf c_{\mathsf d} (H)] \,.
      \]

\item If $G$ is infinite, then $\Delta (H)=\N$, $\daleth^* (H) = \N_{\ge 3}$, and $\Ca_{\mathsf d} (H)=\mathcal R_{\mathsf d} (H)=\N_{\ge 2}$.
\end{enumerate}
\end{theorem}

It remains to study the set of catenary degrees and the set of minimal relations in case $\mathsf D (G) \le 3$. If $H$ is a commutative monoid, then $H$ is factorial if and only if $H$ is a Krull monoid with trivial class group if and only if $\mathsf c (H) = 0$. If $H$ is non-commutative, then in general the catenary degree $\mathsf c_{\mathsf d} (H)$  depends on the distance function $\mathsf d$ (\cite[Remark 4.7]{Ba-Sm15} and \cite[Example 5.7]{Sm16a} for examples). Even for a principal ideal domain $R$ it can happen that $\mathsf c_{\mathsf d} (R^{\bullet})> 0$ (\cite[Example 7.13]{Ba-Sm15}). We do not go into these details but provide the results in case $\mathsf D (G) \le 3$ only for commutative Krull monoids.

\medskip
\begin{proposition} \label{4.a}
Let $H$ be a commutative Krull monoid with  class group $G$ such that every class contains a prime divisor.
\begin{enumerate}
\item If $|G|=1$, then $\Ca (H) = \mathcal R (H)=\emptyset$.
\item If $|G|=2$, then  $\Ca (H) = \mathcal R (H)=\{2\}$.

\item Suppose that $\mathsf D (G)=3$.  If one nonzero  class contains at least two distinct  prime divisors, then $\Ca (H)=\mathcal R (H)=[2,3]$ and otherwise we have $\Ca (H)=\mathcal R (H)=\{3\}$.
\end{enumerate}
\end{proposition}

\smallskip
Let $G$ be a finite abelian group with $|G| \ge 3$, say $G \cong C_{n_1} \oplus \ldots \oplus C_{n_r}$ with $1 < n_1 \mid \ldots \mid n_r$.  By \cite[Theorem 6.4.2]{Ge-HK06a}, we have
\begin{equation} \label{catenary-bounds}
\Bigl\{ n_r,\, 1 + \sum_{i=1}^r \left\lfloor \frac{n_i}{2} \right\rfloor \Bigr\} \le \mathsf c (G) \le \mathsf  D (G)\,.
\end{equation}
Furthermore, we have $\mathsf c (G) = \mathsf D (G)$  if and only if $G$ is either cyclic or an elementary $2$-group (\cite[Theorem 6.4.7]{Ge-HK06a}). The groups $G$ satisfying $\mathsf c (G) \in [3,4]$ or $\mathsf c (G) = \mathsf D (G)-1$ are characterized (\cite{Ge-Zh15b}), but beyond that the precise value of $\mathsf c (G)$ in terms of the group invariants are unknown.
In contrast to the set of (all) distances $\Delta (G)$,
the set of minimal distances $\Delta^* (G) \subset \Delta (G)$ is far from being an interval (e.g., \cite{Pl-Sc18a}). However, there is a characterization of when $\Delta^* (G)$ is an interval (\cite[Theorem 1.1]{Zh18a}).

In general, the  inclusions in \eqref{inclusion1}, \eqref{inclusion2}, and \eqref{inclusion3} can be strict and the sets need not be intervals. Indeed,  even in case of commutative Krull monoids every finite set can be realized as the set of catenary degrees and the same is true for the remaining sets of invariants.  More precisely, we have (see \cite[Proposition 3.2]{Fa-Ge17a} and \cite[Theorem 1.1]{Ge-Sc17a}).

\medskip
\begin{proposition} \label{4.c}
Let $C \subset \N$  be a finite nonempty subset.
\begin{enumerate}
\item There is a finitely generated commutative Krull monoid $H$ with finite class group such that $\mathcal R (H) = \Ca (H) = C \setminus \{1\}$ and $\daleth^* (H) = C \setminus \{1, 2\}$.

\item If $\min C = \gcd C$, then there is a finitely generated commutative Krull monoid $H$ with $\Delta (H)= C$.
\end{enumerate}
\end{proposition}

The above realization result for abstract Krull monoids together with
Claborn's Realization Theorem (\cite[Chapter 3.7]{Ge-HK06a}) and the realization theorem of Facchini and Wiegand (\cite[Theorem 3.4]{Fa-Wi04}) shows that there are Dedekind domains and monoids of modules (as discussed in Example \ref{sec-2.5}.2) whose sets of arithmetical invariants have the given form (details are worked out for the set of distances in \cite[Corollary 1.2]{Ge-Sc17a}).
Apart from such abstract realization theorems, there is a variety of results on these sets for very specific monoids and domains (e.g., \cite{C-C-M-M-P14, GG-MF-VT15, N-P-T-W16a, Ge-Zh16c, Co-Ka17a, ON-Pe18a}, \cite[Theorem 4.11]{Fa-Tr18a}). To provide an explicit example where, say the set of distances, is not an interval, consider the numerical monoid $H$ generated by $\mathcal A (H) = \{n, n+1, n^2-n-1\}$ for some $n \in \N_{\ge 3}$. Then $\Delta (H) = [1, n-2] \cup \{2n-5\}$ (\cite{B-C-K-R06}).
We end this subsection with a list of examples satisfying the assumptions in Theorem \ref{4.b}.

\medskip
\begin{example} \label{4.d}
We provide examples of transfer Krull monoids over finite abelian groups and, in particular, examples of commutative Krull monoids having prime divisors in all classes.

1. (Commutative Krull monoids having prime divisors in all classes)
Let $H$ be a commutative Krull monoid with class group $G$ and let $G_P \subset G$ denote the set of classes containing prime divisors. Then there is a transfer homomorphism $\theta \colon H \to \mathcal B (G_P)$ with $\mathsf c (H, \theta) \le 2$ (\cite[Theorem 3.4.10]{Ge-HK06a}). We continue with explicit examples where $G_P=G$ holds.

(a) The ring of integers $\mathcal O_K$ of an algebraic number field is a Dedekind domain with finite class group and each class contains infinitely many prime divisors, and the same is true for holomorphy rings in global fields. The distribution of prime divisors in the classes of a given Dedekind domain is a  problem which received a lot of attention (see \cite[Theorem 3.7.8]{Ge-HK06a} and the subsequent discussion for a survey).

(b) (Regular congruence monoids) Let $R$ be a commutative Dedekind domain, $\mathfrak f \subset R$ a nonzero ideal,  $\Gamma \subset R/\mathfrak f$ a multiplicatively closed subset, and $H_{\Gamma} = \{ a \in R^{\bullet} \mid a+ \mathfrak f \in \Gamma \}$. If $aR + \mathfrak f = R$ for all $a \in H_{\Gamma}$, then $H_{\Gamma}$ is a regular congruence monoid and hence it is Krull. If every class of $R$ contains a prime divisor, then the same is true for the regular congruence monoids  defined in $R$ (\cite[Chapter 2.11]{Ge-HK06a}).

(c) (Semigroup rings) If $R[H]$ is a Krull monoid domain, then every class contains a prime divisor (\cite{Ch11a}).

(d) (Finitely generated domains) If $R$ is an integral separable finitely generated algebra over an infinite field $K$ with $\dim_K (R) \ge 2$, then $R$ is noetherian and every class of its $v$-class group contains infinitely many prime divisors (\cite{Ka99c}).

(e) Let $G$ be an abelian group with $|G| \ne 2$. Then $\mathcal B (G)$ is a commutative Krull monoid with class group isomorphic to $G$ and each class contains precisely one prime divisor (\cite[Proposition 2.5.6]{Ge-HK06a}). This generalizes to  relative block monoids (\cite{HK92e, Ba-Ho09a}). Let $K \subset G$ be a subgroup and let $\mathcal B_K (G) = \{S \in \mathcal F (G) \mid \sigma (S) \in K \}$. Then $\mathcal B_K (G)$ is Krull, its class group is isomorphic to $G/K$ (apart from the case where $|G|=2$ and $K = \{0\}$), and each class contains precisely $|K|$ prime divisors.

2. (Commutative transfer Krull monoids which are not completely integrally closed) Let $\mathcal O$ be an order in an algebraic number field $K$, $\mathcal O_K$ its ring of integers, and let $\pi \colon \spec (\mathcal O_K) \to \spec (\mathcal O)$ be the natural map defined by $\pi ( \mathfrak p) = \mathfrak p \cap \mathcal O$. If $\mathcal O$ is seminormal, $\pi$ is bijective, and there is an isomorphism $\Pic (\mathcal O) \to \Pic (\mathcal O_K)$, then there is a transfer homomorphism $\theta \colon \mathcal O^{\bullet} \to \mathcal B ( \Pic (\mathcal O))$ with $\mathsf c (\mathcal O^{\bullet}, \theta) \le 2$  (see \cite[Theorem 5.8]{Ge-Ka-Re15a}, where a more general result is given in the setting of weakly Krull monoids).

3. (Non-commutative transfer Krull monoids) Let $\mathcal O_K$ be the ring of integers of an algebraic number field $K$, $A$ a central simple $K$-algebra, and $R$ a classical maximal $\mathcal O_K$-order of $A$. Let $\mathsf d$ be a distance on $R^{\bullet}$ that is invariant under conjugation by normalizing elements. If every stably free left $R$-ideal is free, then there is a transfer homomorphism $\theta \colon R^{\bullet} \to \mathcal B (G)$ with $\mathsf c_{\mathsf d} (R^{\bullet}, \theta) \le 2$, where $G$ is a ray class group of $\mathcal O_K$ (\cite[Theorem 1.1]{Sm13a} and \cite[Corollary 7.11 and Theorem 7.12]{Ba-Sm15}).
\end{example}

\medskip
\subsection{\bf Proof of the Main Results.} \label{sec-4.3}
We start with two  lemmas.

\medskip
\begin{lemma} \label{4.1}
Let $H$ be a \BF-monoid, $B$ a commutative reduced \BF-monoid, $\mathsf d$ a distance on $H$, and $\theta \colon H \to B$ a transfer homomorphism.
\begin{enumerate}
\item For every $a \in H$ we have $\mathsf c \big (\theta (a) \big) \le \mathsf c_{\mathsf d} (a) \le \max \{ \mathsf c \big( \theta (a) \big), \mathsf c_{\mathsf d} (H, \theta) \big\}$.

\item Suppose that  $\mathsf c_{\mathsf d} (H, \theta) \le 2$. Then $\Ca (B) \subset \Ca_{\mathsf d} (H)$, and if $\mathsf c (B) \ge 2$, then $\mathsf c_{\mathsf d} (H) = \mathsf c (B)$.

\item $\Delta (H)=\Delta (B)$ and $\daleth^* (H)=\daleth^* (B)$.
\end{enumerate}
\end{lemma}

\begin{proof}
1. A proof of the upper bound can be found in \cite[Proposition 4.6]{Ba-Sm15}. The lower bound is obvious in our setting (a proof in a much more general setting can be found in \cite[Theorem 2.22]{Fa-Tr18a}).

2. We suppose that  $\Ca (B) \ne \emptyset$ whence we have $\min \Ca (B) \ge 2$. If $b \in B$ with $\mathsf c (b) \ge 2$, then there is an $a \in H$ with $\theta (a) = b$ and the inequalities in 1. imply that $\mathsf c (b) = \mathsf c_{\mathsf d} (a)$. Thus $\Ca (B) \subset \Ca_{\mathsf d} (H)$. If $\mathsf c (B) \ge 2$, then 1. implies that $\mathsf c_{\mathsf d} (H) = \mathsf c (B)$.

3. Since $\mathcal L (H)=\mathcal L (B)$ by \eqref{transfer}, the assertions follow.
\end{proof}

\medskip
\begin{lemma} \label{4.4}
Let $G$ be a finite abelian group with $|G|\ge 3$.
\begin{enumerate}
\item For every $B\in \mathcal B(G)$ with $\mathsf c(B)\ge 4$  there exists $B'\in \mathcal B(G)$ with $|B'|<|B|$ and $\mathsf c(B')\ge \mathsf c(B)-1$.

\item
\[
\Ca(G)=\left\{
\begin{aligned}
&\{3\}, \qquad &&  \mathsf D (G) = 3, \\
&[2, \mathsf c(G)], \qquad && \mathsf D (G) \ge 4 \,.
\end{aligned}
\right.\]
\end{enumerate}
\end{lemma}

\smallskip

\begin{proof}
1. Let $B\in \mathcal B(G)$ with $\mathsf c(B)\ge 4$.
Then there exists an element $B_0\in \mathcal B(G)$ with $|B_0|$ being minimal such that $\mathsf c(B_0)=\mathsf c(B)$. Then $|B_0|\le |B|$ and since it is sufficient to prove the assertion for $B_0$, we may assume that $B_0=B$. We set $d=\mathsf c(B)$ and observe that  $B\in \mathcal B(G\setminus\{0\})$ by the minimality of $|B|$.
By the definition of $\mathsf c(B)$,  there exist two distinct factorizations
\[
z_0 =U_1\cdot \ldots \cdot U_k \in \mathsf Z (B) \quad \text{and} \quad  z_0'=V_1\cdot \ldots \cdot V_{\ell} \in \mathsf Z (B) \,,
\]
where $k\le \ell \in \N$  and $U_1,\ldots, U_k, V_1,\ldots, V_{\ell} \in \mathcal A (G)$, with $\mathsf d(z_0,z_0')=d$ and such that there is no $(d-1)$-chain concatenating $z_0$ and $z_0'$.
Let $g_1, g_2\in G$ with $g_1g_2\t U_1$ and, without loss of generality, we  assume that $g_1g_2\t V_1V_2$. We define
\[
B'=B(g_1+g_2)(g_1g_2)^{-1}, \ U_1'=U_1(g_1+g_2)(g_1g_2)^{-1}, \quad \text{and} \quad V'=V_1V_2(g_1+g_2)(g_1g_2)^{-1} \,.
\]
Then $|B'|<|B |$, $U_1' \in \mathcal A (G)$, and we consider a factorization of $V'$, say  $V'=W_1\cdot \ldots \cdot W_s$, where $s \in \N$,   $W_1,\ldots,W_s \in \mathcal A (G)$, and $g_1+g_2\t W_1$. We obtain two factorizations of $B'$, namely
\[
z =U_1'U_2\cdot \ldots \cdot U_k \in \mathsf Z (B') \quad \text{and} \quad
z' =W_1\cdot \ldots \cdot W_s V_3\cdot \ldots \cdot V_{\ell} \in \mathsf Z (B') \,.
\]
We assert that there is no $(d-2)$-chain of factorizations  concatenating $z$ and $z'$ which implies that $\mathsf c(B')\ge d-1=\mathsf c(B)-1$. If this holds, then 1. is proved.

Assume to the contrary that there exists a $(d-2)$-chain $(z=z_1', \ldots, z_m'=z')$ of factorizations of $B'$ concatenating $z$ and $z'$. For every $i\in[1,m]$, we set
\[
 z_i'=X_{i,1}\cdot \ldots \cdot X_{i,t_i} \in \mathsf Z (B') \,,
\]
where $t_i\in \N$,  $X_{i,1},\ldots, X_{i,t_i} \in \mathcal A (G)$, and with $g_1+g_2\t X_{i,1}$.
Then $X_{i,1}'=X_{i,1}g_1g_2(g_1+g_2)^{-1} \in \mathcal B (G)$ and we fix a factorization
  $Y_{i,1}\cdot \ldots \cdot Y_{i,r_i} \in \mathsf Z (X_{i,1}')$, where $r_i\in [1,2]$ and $Y_{i,1},\ldots, Y_{i,r_i} \in \mathcal A (G)$.
Thus, for each $i\in [1,m]$, we obtain a factorization
\[
z_{2i} =Y_{i,1}\cdot \ldots \cdot Y_{i,r_i}X_{i,2}\cdot \ldots \cdot X_{i,t_i} \in \mathsf Z (B) \,.
\]
We choose $i\in [1,m-1]$ and distinguish three cases.

First, if there exists $j_0\in[1, t_{i+1}]$ such that $X_{i,1}=X_{i+1,j_0}$, then we set
\[
z_{2i+1} =Y_{i,1}\cdot \ldots \cdot Y_{i,r_i}\cdot \prod_{j\in[1, t_{i+1}]\setminus \{j_0\}}X_{i+1, j} \in \mathsf Z (B) \,,
\]
which implies that  $\mathsf d(z_{2i}, z_{2i+1})\le d-2\le d-1$ and $\mathsf d(z_{2i+1}, z_{2i+2})\le 3\le d-1$.

Second, if there exists $j_0\in[1, t_i]$ such that $X_{i,j_0}=X_{i+1,1}$, then we set
\[
z_{2i+1}=Y_{i+1,1}\cdot \ldots \cdot Y_{i+1,r_{i+1}}\cdot \prod_{j\in[1, t_i]\setminus \{j_0\}}X_{i, j} \in \mathsf Z (B) \,,
\]
which implies that  $\mathsf d(z_{2i}, z_{2i+1})\le 3\le d-1$ and $\mathsf d(z_{2i+1}, z_{2i+2})\le d-2\le d-1$.

Finally, if none of the previous two conditions holds, then $\mathsf d(z_{2i}, z_{2i+2})\le d-2+1=d-1$ and we set $z_{2i+1}=z_{2i}$.

Clearly, we have $\mathsf d(z_0, z_2)\le 3\le d-1$ and thus there is a $(d-1)$-chain concatenating $z_0$ and
\[
  z_{2m}=Y_{m,1}\cdot \ldots \cdot Y_{m,r_m}W_2\cdot \ldots \cdot W_s V_3\cdot \ldots \cdot V_{\ell} \in \mathsf Z (B) \,.
\]
 Thus $V_1V_2=Y_{m,1}\cdot \ldots \cdot Y_{m,r_m}W_2\cdot \ldots \cdot W_s$. Since $\ell\ge \mathsf c(B)\ge 4$, we obtain $|V_1V_2|<|B|$ and hence $\mathsf c(V_1V_2)\le d-1$ by the minimality of $|B|$. Therefore there is a $(d-1)$-chain concatenating $z_{2m}$ and $z_0'$ and hence there is a $(d-1)$-chain concatenating $z_0$ and $z_0'$, a contradiction.

\smallskip
2. If $\mathsf D (G)=3$, then  \eqref{catenary-bounds} shows that $\mathsf c (G)=3$ and it easily follows that $\Ca(G)=\{3\}$.
Suppose that $\mathsf D (G)\ge 4$. Then \eqref{Davenport1} and  \eqref{catenary-bounds}  imply that $\mathsf c (G) \ge 4$. Let $B\in \mathcal B(G)$ such that $\mathsf c(B)=\mathsf c(G)$. Then 1. implies that there exist $B=B_0, B_1,\ldots B_k\in \mathcal B(G)$, where $k<|B|$,  such that $\mathsf c(B_i)\ge \mathsf c(B_{i+1})-1$ for each $i\in [0,k-1]$ and $\mathsf c(B_k)=3$.
It remains to verify that $2 \in \Ca (G)$ and we distinguish three cases.

If there is an element $g\in G$ with $\ord(g)=n \ge 4$, then $2= \mathsf c \big( g^n((n-2)g)(2g) \big) \in \Ca (G)$.

If there are two independent elements $e_1, e_2 \in G$ with $\ord (e_1)=\ord (e_2)=3$, then $2 = \mathsf c \big( (e_1+e_2)^4e_1^2e_2^2 \big) \in \Ca (G)$.

If there are three distinct elements $e_1, e_2, e_3$ with $\ord (e_1)=\ord (e_2)=\ord (e_3)=2$, then  $2 = \mathsf c \big(  e_1e_2e_3(e_1+e_2+e_3)(e_1+e_2)^2  \big) \in \Ca (G)$.
\end{proof}

\medskip
\begin{proof}[Proof of Theorem \ref{4.b}]
Let  $\theta \colon H \to \mathcal B (G)$ be a transfer homomorphism to an abelian group $G$  and let  $\mathsf d$ be a distance on $H$ with $\mathsf c_{\mathsf d} (H, \theta) \le 2$. By Lemma \ref{4.1}.2, we have $\Ca (G) \subset \Ca_{\mathsf d} (H)$ and, if $\mathsf c (G)\ge 2$, then $\mathsf c (G) = \mathsf c_{\mathsf d} (H)$. By \eqref{transfer}, we have $\mathcal L (H)=\mathcal L (G)$ whence $\Delta (H)=\Delta (G)$ and $\daleth^* (H)=\daleth^* (G)$.

1.  If $\mathsf D (G) \le 2$, then $\mathcal B (G)$ is half-factorial  whence $\Delta (G)=\daleth^* (G)=\emptyset$.
If $\mathsf D (G) = 3$, then $\Delta (G)$ and $\daleth^* (G)$ have the asserted values by  \eqref{distance-catenary} and \eqref{catenary-bounds}.

2. Suppose that $G$ is finite with $\mathsf D (G)\ge 4$. Then $\mathsf c_{\mathsf d} (H) = \mathsf c (G) \in [3,  \mathsf D (G)]$ by \eqref{catenary-bounds}.
Since $\mathcal B (G)$ is a Krull monoid with class group isomorphic to $G$ and each class contains a prime divisor, $\Delta (G)$ is an interval with $\min \Delta (G)=1$ by \cite[Theorem 1.1]{Ge-Yu12b}. The set $\daleth^* (G)$ is an interval with $\min \daleth^* (G)=3$ by \cite[Proposition 3.3]{Fa-Ge17a}.
By \eqref{inclusion1} and \eqref{distance-catenary}, we have $\daleth^* (G) \subset 2 + \Delta (G) \subset [2, \mathsf c (G)]$.
Thus by Lemma \ref{4.4}.2, we obtain that
\[
[2, \mathsf c_{\mathsf d} (H)] = [2, \mathsf c (G)] = \Ca (G) \subset \Ca_{\mathsf d} (H) \subset \mathcal R_{\mathsf d} (H) \subset [2, \mathsf c_{\mathsf d} (H)] \,,
\]
whence we have equality throughout. If $\mathsf D (G) = \mathsf D^* (G)$, then $\max \daleth^* (G) = 2+\max \Delta (G)= \mathsf c (G)$ by \cite[Corollary 4.1]{Ge-Gr-Sc11a}. Since $\Delta (H)$ and $\daleth^* (H)$ are intervals, we infer that
\[
\Ca_{\mathsf d} (H) = [2, \mathsf c_{\mathsf d} (H)] = \daleth^* (H) \cup \{2\} =  \big(2+\Delta (H)\big) \cup \{2\}   =   [2, \mathsf c_{\mathsf d} (H)] \,.
\]

3. Suppose that $G$ is infinite. Then for every finite set $L \subset \N_{\ge 2}$ there is an $a \in H$ such that $\mathsf L (a) = L$ (\cite[Theorem 7.4.1]{Ge-HK06a}). This implies that $\Delta (H)=\N$ and $\daleth^* (H) = \N_{\ge 3}$. Since $\Ca (G) \subset \Ca_{\mathsf d} (H) \subset \mathcal R_{\mathsf d} (H)$, it remains to show that $\N_{\ge 2} \subset \Ca (G)$.

Suppose that  $G$ is an infinite torsion group. Then \cite[Lemma 6.4.1]{Ge-HK06a} implies that for every $k\in \N_{\ge 3}$, there exists $A_k\in \mathcal B(G)$ such that $\mathsf c(A_k)=k$. It follows by Lemma \ref{4.4}.2 that $2\in \Ca(G)$. Therefore $\N_{\ge 2}\subset\Ca(G)$.

Suppose  there exists an element $g\in G$ with $\ord(g)=\infty$. Then for every $n\in \N_{\ge 2}$, we set $U=g^n(-ng)$, $V=g(-g)$, and $V'=ng(-ng)$. It follows that $\mathsf Z(g^n(-g)^nng(-ng))=\{U (-U), V^n V'\}$ which implies that $\mathsf c(g^n(-g)^nng(-ng))=n+1$ whence $\N_{\ge 3}\subset \Ca(G)$.
If $W=(-g)^2(-4g)^2(2g)^2(3g)^2$, then
\[
\mathsf Z(W)=\{((-g)(-4g)2g3g)^2, (-g)^22g\cdot (-4g)^22g(3g)^2, (-4g)(2g)^2\cdot (-g)^2(-4g)(3g)^2\}
\]
which implies that $2=\mathsf c(W)\in \Ca(G)$.
\end{proof}

\medskip
\begin{proof}[Proof of Proposition \ref{4.a}]~

1.  If $|G|=1$, then $H$ is factorial whence $\mathsf c (H)=0$ and $\Ca (H) = \mathcal R (H)=\emptyset$.

\smallskip
2. Suppose that $|G|=2$.   Since $H$ is not factorial, we obtain that $0 < \mathsf c (H) \le \mathsf D (G)=2$ by \eqref {catenary-bounds} whence $\Ca (H) = \{2\}$. Since $\Ca (H) \subset \mathcal R (H)$ and the maxima of the sets coincide by \eqref{inclusion2}, it follows that $\mathcal R (H)=\{2\}$.

\smallskip
3. Suppose that  $\mathsf D (G)=3$. Then $\mathsf c (G) =3$ by \eqref{catenary-bounds} whence $\Delta (H)=\{1\}$ and $\daleth^* (H)=\{3\}$. If there is a nonzero class containing at least two
distinct prime divisors, then $\min \Ca (H) = \min \mathcal R (H) = 2$ hence $\Ca (H)=\mathcal R (H)=[2,3]$ and otherwise we have $\mathcal R (H)=\Ca (H) = \{3\} $ by \cite[Proposition 3.4]{Fa-Ge17a}.
\end{proof}

\medskip
\section{The set of tame degrees} \label{5}
\medskip

Local tameness (i.e., the finiteness of all local tame degrees) is a basic finiteness property (in commutative factorization theory) in the sense that in many settings local tameness has to be guaranteed first before one proceeds to establish further arithmetical finiteness properties (we refer to the proof of the Structure Theorem for sets of lengths \cite[Chapter 4.3]{Ge-HK06a} which serves as a prototype for this procedure). Tame degrees have been introduced also in the non-commutative setting (\cite[Section 5]{Ba-Sm15}) but have not yet proved their usefulness in that context. Thus in this section we restrict to commutative Krull monoids. We introduce tame degrees in Subsection \ref{sec-5.1}, formulate our main results in Subsection \ref{sec-5.2}, and prove them in Subsection \ref{sec-5.3}.

\subsection{Tame degrees} \label{sec-5.1}
Let $H$ be commutative  \BF-monoid. The {\it tame degree} $\mathsf t (a,u)$ of an element $a \in H$ and an atom $u \in H$ is the smallest integer $N$ with the following property: if $a \in uH$, then for any factorization $a=v_1 \cdot \ldots \cdot v_n$, with  $v_1, \ldots, v_n \in \mathcal A (H)$, there is a subproduct which is a multiple of $u$, say $v_1 \cdot \ldots \cdot v_m$, and a refactorization of this subproduct which contains $u$, say $v_1 \cdot \ldots \cdot v_m = uu_2 \cdot \ldots \cdot u_{\ell}$ such that $\max \{\ell, m\} \le N$. More formally,
for $a \in H$ and $u \in \mathcal A (H_{\red})$,  $\mathsf t_H (a,u) = \mathsf t (a,u)$ is the smallest $N \in \mathbb N_0 \cup \{\infty\}$ having the following property:
\begin{itemize}
\item If $\mathsf Z (a) \cap u \mathsf Z (H) \ne \emptyset$ and $z \in \mathsf Z (a)$, then there exists some $z' \in \mathsf Z (a) \cap u \mathsf Z (H)$ such that $\mathsf d (z,z') \le N$.
\end{itemize}
By convention, $\mathsf Z (a) \cap u \mathsf Z (H) = \emptyset$ implies that $\mathsf t (a,u)=0$. If $\mathsf Z (a) \cap u \mathsf Z (H) \ne \emptyset$, then $\mathsf t (a,u) = 0$ if and only if $\mathsf Z (a) \cap u \mathsf Z (H) = \mathsf Z (a)$ (in other words, if every factorization of $a$ is divisible by $u$).
Since $\mathsf t (a, u) \le \max \mathsf L (a)$, all tame degrees are finite.
We call
\[
\Ta (H) = \{ \mathsf t (a,u) \mid a \in H, u \in \mathcal A (H_{\red}), \mathsf t (a,u)>0 \} \subset \N_0
\]
the {\it set of tame degrees} of $H$, $\mathsf t (H,u) = \sup \{\mathsf t (a,u) \mid a \in H \}$
is the {\it local tame degree} (at $u$), and
\[
\mathsf t (H)  = \sup \Ta (H) = \sup \{ \mathsf t (H, u) \mid u \in \mathcal A (H_{\red}) \} \in \N_0 \cup \{\infty\}
\]
is the {\it $($global$)$ tame degree} of $H$ (with the convention that $\sup \emptyset = 0$). We say that  $H$ is
\begin{itemize}
\item {\it locally tame} if $\mathsf t (H,u) < \infty$ for all $u \in \mathcal A (H_{\red})$, and

\item {\it (globally) tame} if $\mathsf t (H) < \infty$.
\end{itemize}
It is easy to check that $\mathsf t (H,u)=0$ if and only if $u$ is a prime whence $H$ is factorial if and only if $\mathsf t (H)=0$. Furthermore, we have (\cite[Theorem 1.6.6]{Ge-HK06a})
\[
\mathsf c (H) \le \mathsf t (H), \quad \text{and if $H$ is not factorial, then} \quad \max \{2, \rho (H)\} \le \mathsf t (H) \,.
\]
If $G$ is an abelian group, $G_0 \subset G$ a subset, $A \in \mathcal B (G_0)$, and $U \in \mathcal A (G_0)$, then we set (as usual)
\[
\mathsf t (G_0, U) = \mathsf t \big( \mathcal B (G_0), U \big), \quad \mathsf t (G_0) = \mathsf t \big( \mathcal B (G_0) \big), \quad \text{and} \quad  \Ta (G_0) = \Ta \big( \mathcal B (G_0) \big) \,.
\]

\subsection{Main Results} \label{sec-5.2} We formulate the main results and discuss them afterwards.

\medskip
\begin{theorem} \label{5.1}
Let $H$ be a commutative Krull monoid whose class group $G$ is an elementary $2$-group, say $G = C_2^r$ with $r \in \N_0$, and suppose that every class contains a prime divisor.
\begin{enumerate}
\item If $r=0$, then $\Ta (H)=\Ta(G)=\emptyset$ and if $r=1$, then $\Ta (H) = \{2\}$ and $\Ta (G) = \emptyset$.

\item If $r=2$, then $\Ta (G) = \{3\}$, and if one nonzero class contains at least two distinct prime divisors, then $\Ta (H) = [2,3]$.

\item If $r=3$, then $\Ta (G) = [2,4]$, and if one nonzero class contains at least two distinct prime divisors, then $\Ta (H) = [2,5]$.

\item \[
      \Ta (H) \quad \begin{cases}  = \Ta (G) = [2, 1 + \frac{r^2}{2}] & \quad \text{if} \ \ r \ge 4 \quad \text{is even,}  \\
                                   \supset \Ta (G)           \supset [2,  2 + \frac{r(r-1)}{2}] & \quad \text{if} \ \  r \ge 5 \quad \text{is odd.}
      \end{cases}
      \]
\end{enumerate}
\end{theorem}

\medskip
\begin{theorem} \label{5.2}
Let $H$ be a commutative Krull monoid with class group $G$ such that $\mathsf D (G) \ge 3$ and suppose that every class contains a prime divisor.
\begin{enumerate}
\item $\Ta (C_3)= \Ta (C_2 \oplus C_2)=\{3\}$. If $G$ is finite and either $\mathsf D (G) \ge 4$ or there is a nonzero class containing at least two distinct prime divisors, then $[2, \mathsf D (G)] \subset \Ta (H)$.

\item If every class contains at least  $\mathsf D(G)+1$ prime divisors, then $\Ta (H) = [2, \mathsf t (H)]$.

\item If  $G$ is infinite, then $\Ta (H) = \N_{\ge 2}$.
\end{enumerate}
\end{theorem}

Let $H$ be a commutative Krull monoid with class group $G$ and let $G_P \subset G$ denote the set of classes containing prime divisors. The relationship between the tame degrees of $H$ and  the tame degrees of $\mathcal B (G_P)$ is not as close as it was with the catenary degrees. We will have again the basic inclusion  $\Ta (G_P) \subset \Ta (G)$ (Lemma \ref{5.3}).
However, it is not true that $H$ is globally tame (i.e., $\sup \Ta (H) < \infty$) if and only if $\mathcal B (G_P)$ is globally tame (i.e., $\sup \Ta (G_P) < \infty$; see \cite[Remark 3.4]{Ga-Ge-Sc15a}). Moreover, it is an open problem whether for all finite abelian groups $G$ with $\mathsf D (G) \ge 5$ (or with $\mathsf D (G)$ being sufficiently large), we have $\mathsf t (H) = \mathsf t (G)$. Theorem \ref{5.1}.3 reveals  that this fails for $\mathsf D (G)=4$. Proposition \ref{5.5} shows that every finite nonempty subset can be realized as the set of tame degrees of a commutative Krull monoid, if we do not impose any assumption on the distribution of prime divisors.
The standing conjecture for $G=C_2^r$ is that, for odd $r \ge 5$, we have $\mathsf t (H)=2+r(r-1)/2$. If this holds true, then all inclusions in Theorem \ref{5.1}.4. are equalities. The precise value of $\mathsf t (G)$ (in terms of the group invariants) is unknown  apart from the groups given in Theorem \ref{5.1} and a couple of small groups.

Let $G$ be an abelian group. Then $\mathcal B (G)$ is locally tame if and only if $G$ is finite (\cite[Theorem 4.4]{Ge-Ha08a}). If $G$ is finite with $|G| \ge 3$, then for all $U \in \mathcal A (G)$ we have the bounds (\cite[Proposition 6.5.1]{Ge-HK06a})
\begin{equation} \label{tamedegree-bounds}
\mathsf t (G, U) \le 1 + \frac{ |U|\, (\mathsf D (G)-1)}2 \quad \text{and} \quad \mathsf D (G) \le \mathsf t (G) \le 1 + \frac{\mathsf D (G) (\mathsf D (G)-1)}2\,.
\end{equation}

\subsection{Proof of the Main Results} \label{sec-5.3} We start with some lemmas.

\medskip
\begin{lemma} \label{5.3}
Let $H$ be a commutative Krull monoid with class group $G$ and let $G_P \subset G$ denote the set of classes containing prime divisors.
\begin{enumerate}
\item $H$ is locally tame if and only if $\mathcal B (G_P)$ is locally tame.

\item $\Ta (G_P) \subset \Ta (H)$ and hence  $\mathsf t (G_P) \le \mathsf t (H)$.

\item If there is a nonzero class containing at least two distinct prime divisors, then $2 \in \Ta (H)$.
\end{enumerate}
\end{lemma}

\begin{proof}
We may suppose that $H$ is reduced. Let the inclusion $H \hookrightarrow F = \mathcal F (P)$ be a divisor theory of $H$. Then $G = \mathsf q (F)/\mathsf q (H) = \{[a]= a \mathsf q (H) \mid a \in F \}$ is the class group of $H$ and $G_P = \{[p] \mid p \in P \}$ is the set of classes containing prime divisors. Let $\boldsymbol \beta \colon H \to \mathcal B (G_P)$ be the transfer homomorphism defined by $\boldsymbol \beta ( p_1 \cdot \ldots \cdot p_{\ell}) = [p_1] \cdot \ldots \cdot [p_{\ell}]$, where $p_1, \ldots, p_{\ell} \in P$ (see \cite[Proposition 3.4.8]{Ge-HK06a}).

1. See \cite[Proposition 3.3]{Ga-Ge-Sc15a}.

2.  Let $A \in \mathcal B (G_P)$ and $U \in \mathcal A (G_P)$ be given with $\mathsf t (A, U) > 0$. Thus $U \t A$ and $U \ne A$, say $U = g_1 \cdot \ldots \cdot g_k$ and $A = U g_{k+1} \cdot \ldots \cdot g_{\ell}$ with $\ell \in \N$ and $k \in [1, \ell-1]$. For every $i \in [1, \ell]$ we choose a prime element $p_i \in g_i \cap P$, and we do it in such a way that $g_i=g_j$ implies that $p_i=p_j$ for all $i, j \in [1, \ell]$. Then $u=p_1 \cdot \ldots \cdot p_k \in \mathcal A (H)$, $a = u p_{k+1} \cdot \ldots \cdot p_{\ell} \in H$, $\boldsymbol \beta (u)=U$, and $\boldsymbol \beta (a) = A$. Clearly, there is a bijection from $\mathsf Z_H (a)$ to $\mathsf Z_{\mathcal B (G_P)} (A)$, and this implies that $\mathsf t (A,U)= \mathsf t (a,u) \in \Ta (H)$.

3. Let $g \in G_P$ with $\ord (g) = n \ge 2$ and $p, q \in g$ be two distinct prime divisors. Then $u = p^n, v = q^n \in \mathcal A (H)$ and for $a = u v = (p^{n-1}q) (p q^{n-1})$ we have $\mathsf t (a, u)=2$.
\end{proof}

\medskip
\begin{lemma} \label{5.4}
Let $n\in \N$,   $(H_i)_{i=1}^n$ be a family of commutative \BF-monoids, and $H=H_1\times\ldots\times H_n$. Then $\Ta(H)=\bigcup_{i=1}^n\Ta(H_i)$.
\end{lemma}

\begin{proof}
Without loss of generality, we may assume that $H_1,\ldots, H_n$ are reduced. Then $H_1,\ldots , H_n$ are divisor closed
submonoids of $H$ whence $\bigcup_{i=1}^n\Ta(H_i)\subset \Ta(H)$. Thus we have to verify the reverse inclusion, and by an inductive argument it suffices to do this for $n=2$.

Let $a\in H$ and $u\in \mathcal A(H)$ with $\mathsf t_H (a,u)> 0$.   Therefore $u \t a$ (in $H$) and there exist $a_1\in H_1$ and $a_2\in H_2$ such that $a=a_1a_2$. Note that $\mathcal  A(H)=\mathcal A(H_1)\cup\mathcal A(H_2)$ and $\mathsf Z_H (a) = \mathsf Z_{H_1} (a_1) \mathsf Z_{H_2} (a_2)$. Thus there exists $i\in [1,2]$,  such that $u\in \mathcal A(H_i)$, say $i=1$. Since $u \t a$ in $H$, it follows that $u \t a_1$ in $H_1$ whence $\mathsf t_H(a, u)=\mathsf t_{H_1} (a_1,u)\in \Ta(H_1)$.
\end{proof}

\begin{proposition} \label{5.5}
For every finite nonempty subset $C \subset \N_{\ge 2}$ there is a finitely generated commutative Krull monoid $H$ with finite class group such that
$\Ta(H)=C$.
\end{proposition}

\begin{proof}
Let $C=\{d_1,\ldots, d_n\} \subset \N_{\ge 2}$  be a finite nonempty subset with $n\in \N$ and $d_1,\ldots d_n\in \N_{\ge 2}$.
Let $G$ be a finite abelian group with $G\cong C_{d_1}\oplus\ldots \oplus C_{d_n}$ and $(e_1,\ldots, e_n)$ be a basis of $G$ with $\ord(e_i)=d_i$ for all $i \in [1,n]$.
We define $H_i=\mathcal B(\{e_i,-e_i\})$ for all $i\in [1,n]$.  Then $H=H_1\times\ldots \times H_n$ is a finitely generated Krull monoid with finite class group, and by Lemma \ref{5.4}, it is sufficient to  prove that $\Ta(H_i)=\{d_i\}$ for each $i\in [1,n]$.

Let $i\in [1,n]$. Then $\mathcal A(H_i)=\{U_i=e_i^{d_i}, -U_i=(-e_i)^{d_i}, V_i=e_i(-e_i)\}$, $U_i(-U_i)= V_i^{d_i}$  and hence $ \{d_i\} = \mathsf t \big( U_i(-U_i), U_i) \in \Ta (H_i)$. It remains to show that for all $A \in H_i$ we have $\mathsf t (A, U_i), \mathsf t (A, -U_i), \mathsf t (A, V_i) \in \{0, d_i\}$.

Let $A =e_i^s(-e_i)^t\in \mathcal F (\{e_i, -e_i\})$ with $s, t \in \N_0$. Then $A \in H_i$ if and only if  $s\equiv t \pmod {d_i}$.
If this holds and  $s_0\in [0, d_i-1]$ such that $t\equiv s_0\pmod {d_i}$, then
\[
\mathsf Z_{H_i} (A) =  \left\{ U_i^{\frac{s-s_0}{d_i}-j}(-U_i)^{\frac{t-s_0}{d_i}-j}  V_i^{s_0+jd_i} \ \Big| \quad  j\in \left[0, \min\left\{\frac{s-s_0}{d_i}, \frac{t-s_0}{d_i}\right\}\right]  \right\}
\]
which shows the assertion.
\end{proof}

\medskip
\begin{lemma} \label{5.6}
Let $G$ be an  abelian group. If $G$ is finite, then  $[3, \mathsf D (G)] \subset \Ta (G)$ and if $G$ is infinite, then $\N_{\ge 3} \subset \Ta (G)$.
\end{lemma}

\begin{proof}
Let $m \in \N_{\ge 3}$, and if $G$ is finite, we suppose that $m \le \mathsf D (G)$. It is sufficient to show that $m \in \Ta (G)$.

Clearly, there is a $U \in \mathcal A (G)$ with $|U|=m$, say $U = g_1 \cdot \ldots \cdot g_m$. For $i \in [1,m]$, we set $V_i = (-g_i)g_i$. We consider the element $A = (-U)U \in \mathcal B (G)$ and claim that $\mathsf t (A,U)=m$. Obviously,  $Z'=(-U)U$ is the only factorization of $A$ divisible by $U$. Thus, if $Z= V_1 \cdot \ldots \cdot V_m \in \mathsf Z (A)$, then $\mathsf d (Z,Z')=m$. Since every factorization $Z'' \in \mathsf Z (A)$ not divisible by $U$ has length $|Z''| \le m$, it follows that $\mathsf d (Z'',Z') \le \max \{|Z''|,|Z'|\} \le m$. Thus we obtain that $\mathsf t (A,U)=m$.
\end{proof}

\medskip
\begin{proof}[Proof of Theorem \ref{5.1}]~
If $r=0$, then $|G|=1$ and both $H$ and $\mathcal B (G)$ are factorial whence $\mathsf t (H)=\mathsf t (G)=0$ and $\Ta (H)=\Ta (G)=\emptyset$. Let $r \ge 1$,
$(e_1, \ldots, e_r)$ be a basis of $G = C_2^r$, and let $e_0 = e_1+ \ldots + e_r$. We use that $\mathsf D (G)=r+1$, $\mathsf t (H) \in \Ta (H)$, and that
\[
[3, \mathsf D (G)] \subset \Ta (G) \subset \Ta (H) \,.
\]
If there is a nonzero class  having at least two distinct prime divisors, then $2 \in \Ta (H)$ by Lemma \ref{5.3}.3.
For every $s \in [1,r]$, $\mathcal B (C_2^s)$ is a divisor-closed submonoid of $\mathcal B (C_2^r)$ whence $\Ta (C_2^s) \subset \Ta (G)$. By \cite[Theorem 5.1]{Ga-Ge-Sc15a}, we have
\begin{itemize}
\item If $r=1$, then $\mathsf t (H) = 2$ and $\mathsf t (G) = 0$.

\item If $r=2$, then $\mathsf t (H)=\mathsf t (G)=3$.

\item If $r=3$, then $\mathsf t (G) = 4$, and if  one nonzero class contains at least two distinct prime divisors, then $\mathsf t (H) = 5$.
\end{itemize}

\smallskip
1. If $r=1$, then $\mathsf t (H)=2$ whence $\Ta (H) = \{2\}$, and $\mathcal B (G)$ is factorial whence $\mathsf t (G) = 0$ and $\Ta (G) = \emptyset$.

\smallskip
2. Suppose that $r=2$. Then $\mathsf t (H) = \mathsf t (G) = 3$. It is easy to see that $\Ta (G) = \{3\}$. If $H$ is a commutative Krull monoid having at least two distinct prime divisors in a nonzero class, then $\Ta (H) = [2,3]$.

\smallskip
3. Suppose that $r = 3$. We set
\[
U = e_0e_1 e_2 e_3, \
W = (e_1+e_2)^2, \ X = e_0(e_1+e_2)e_3, \ \text{ and } \ Y = (e_1+e_2)e_1e_2 \,.
\]
Then $A = UW = XY$ and $\mathsf t (A,U)=2$ which implies that $\min \Ta (G) = 2$. Since $\mathsf t (G)=4$ and $\Ta (C_2^2) \subset \Ta (G)$, it follows that $\Ta (G)=[2,4]$. Suppose that $H$ is a commutative Krull monoid having at least two distinct prime divisors in a nonzero class, then $\mathsf t (H) = 5$ and hence $\Ta (G) = [2,4]$ implies that $\Ta (H) = [2,5]$.

\smallskip
4. We continue with the following  assertions.

\noindent

\begin{enumerate}
\item[{\bf A1.}\,] Let $r \ge 4$ be even. Then there are a $U \in \mathcal A (G)$ and, for every $\nu \in [1,r]$, an $A^{(\nu)} \in \mathcal B (G)$ such that $\mathsf t (A^{(\nu)}, U) = 2 - \nu + r^2/2$.

\item[{\bf A2.}\,] Let $r \ge 5$ be odd. Then there are a $U \in \mathcal A (G)$ and, for every $\nu \in [1,r-1]$, an $A^{(\nu)} \in \mathcal B (G)$ such that $\mathsf t (A^{(\nu)}, U) = 2 - \nu + r (r-1)/2$.

\item[{\bf A3.}\,.] Let $r \ge 4$ be congruent two modulo $4$. Then there are a $U \in \mathcal A (G)$ and, for every $\nu \in [1,r-2]$, an $A^{(\nu)} \in \mathcal B (G)$ such that $\mathsf t (A^{(\nu)}, U) = 3 - \nu + r (r-2)/2$.

\item[{\bf A4.}\,.] Let $r \ge 3$ be divisible by $4$.     Then there are a $U \in \mathcal A (G)$ and, for every $\nu \in [1,r-2]$, an $A^{(\nu)} \in \mathcal B (G)$ such that $\mathsf t (A^{(\nu)}, U) = 2 - \nu + r (r-2)/2$.
\end{enumerate}

{\it Proof of}\, {\bf A1}.\,  We set $U = e_0 \prod_{j=1}^r (e_0+e_j)$ and, for all $i, \nu \in [1,r]$, we define
\[
\begin{aligned}
U_i = e_i^2, \quad  V_i & = (e_0+e_i)e_i^{-1} \prod_{j=1}^r e_j , \quad  V_0^{(\nu)} = e_0 (e_1+ \ldots + e_{\nu}) \prod_{j=\nu+1}^r e_j ,  \\
 W^{(\nu)} & = (e_1 + \ldots + e_{\nu})e_1 \cdot \ldots \cdot e_{\nu},  \quad \text{and} \quad A^{(\nu)} = V_0^{(\nu)} \prod_{j=1}^r V_j \,.
\end{aligned}
\]
Then $U, V_0^{(\nu)}, W^{(\nu)}, V_1, \ldots, V_r, U_1, \ldots, U_r \in \mathcal A (G)$,
\[
z = V_0^{(\nu)} \prod_{j=1}^r V_j \in \mathsf Z (A^{(\nu)}), \quad z' = U W^{(\nu)} \prod_{j=1}^{\nu} U_j^{(r-2)/2} \prod_{j=\nu+1}^r U_j^{r/2} \in \mathsf Z (A^{(\nu)}) \,,
\]
and $z'$ is the only factorization of $A^{(\nu)}$ which is divisible by $U$. Therefore we obtain that
\[
\mathsf t (A^{(\nu)}, U) = \mathsf d (z,z') = |z'| = 2 + \nu (r-2)/2 + (r-\nu)r/2  = 2 - \nu + r^2/2 \,.
\]
 \qed[Proof of {\bf A1.}]

\smallskip
{\it Proof of}\, {\bf A2}.\, We set $U =  \prod_{j=1}^r (e_0+e_j)$ and, for all $i \in [1,r]$ and all $\nu \in [1, r-1]$, we define
\[
\begin{aligned}
U_i = e_i^2, \quad  V_i & = (e_0+e_i)e_i^{-1} \prod_{j=1}^r e_j , \quad  V_0^{(\nu)} = (e_0 + e_r) (e_1+ \ldots + e_{\nu}) \prod_{j=\nu+1}^{r-1} e_j ,  \\
 W^{(\nu)} & = (e_1 + \ldots + e_{\nu})e_1 \cdot \ldots \cdot e_{\nu},  \quad \text{and} \quad A^{(\nu)} = V_0^{(\nu)} \prod_{j=1}^{r-1} V_j \,.
\end{aligned}
\]
Then $U, V_0^{(\nu)}, W^{(\nu)}, V_1, \ldots, V_r, U_1, \ldots, U_r \in \mathcal A (G)$,
\[
z = V_0^{(\nu)} \prod_{j=1}^{r-1} V_j \in \mathsf Z (A^{(\nu)}), \quad z' = U W^{(\nu)} \prod_{j=1}^{\nu} U_j^{(r-3)/2} \prod_{j=\nu+1}^r U_j^{(r-1)/2} \in \mathsf Z (A^{(\nu)}) \,,
\]
and $z'$ is the only factorization of $A^{(\nu)}$ which is divisible by $U$. Therefore we obtain that
\[
\mathsf t (A^{(\nu)}, U) = \mathsf d (z,z') = |z'| = 2 + \nu (r-3)/2 + (r-\nu)(r-1)/2  = 2 - \nu + r (r-1)/2 \,.
\]
 \qed[Proof of {\bf A2.}]

\smallskip
{\it Proof of}\, {\bf A3}.\, Consider $G$ as an $\F_2$-vector space and set $U_i = e_i^2$ for all $i \in [1,r]$. Obviously,
\[
U = e_0(e_0+e_1+e_2)(e_0+e_2+e_3) \cdot \ldots \cdot (e_0+e_{r-2}+e_{r-1})(e_0+e_{r-1}+e_1)
\]
is a zero-sum sequence of length $|U|=r$. Since
\[
(e_0+e_1+e_2, e_0+e_2+e_3, \ldots, e_0+e_{r-2}+e_{r-1}, e_0+e_{r-1}+e_1) = (e_1,
\ldots , e_r) \cdot A
\]
and
\[
A = \left( \begin{array}{ccccccc}
0 & 1 & 1 & 1 & \dots & 1 & 0 \\
0 & 0 & 1 & 1 & \dots & 1 & 1 \\
1 & 0 & 0 & 1 & \dots & 1 & 1 \\
1 & 1 & 0 & 0 & \dots & 1 & 1 \\
\vdots& \vdots &\vdots & \dots& \vdots& \vdots \\
1 & 1 & 1 & 1 & \dots & 0 & 0 \\
1 & 1 & 1 & 1 & \dots & 1 & 1
     \end{array} \right)  \quad \in M_{r, r-1} (\F_2)
\]
has rank $r-1$, it follows that $U \in \mathcal A (G)$. For all $i \in [1,r]$ and all $\nu \in [1, r-2]$, we define
\[
\begin{aligned}
V_{\nu}  & = (e_0+e_{\nu}+e_{\nu+1})e_{\nu}^{-1} e_{\nu+1}^{-1} \prod_{j=1}^r e_j , \quad  V_{r-1}  = (e_0+e_{r-1}+e_{1})e_{r-1}^{-1} e_{1}^{-1} \prod_{j=1}^r e_j, \\  V_0^{(\nu)} & = e_0  (e_1+ \ldots + e_{\nu}) \prod_{j=\nu+1}^{r} e_j ,  \qquad
 W^{(\nu)}  = (e_1 + \ldots + e_{\nu})e_1 \cdot \ldots \cdot e_{\nu},   \\ A^{(\nu)} & = V_0^{(\nu)} \prod_{j =1}^{r-1} V_j \,, \quad \text{and} \quad U_i = e_i^2 \,.
\end{aligned}
\]
Then $V_0^{(\nu)}, W^{(\nu)}, V_1, \ldots, V_{r-1}, U_1, \ldots, U_r \in \mathcal A (G)$,
\[
z = V_0^{(\nu)} \prod_{j=1}^{r-1} V_j \in \mathsf Z (A^{(\nu)}), \quad z' = U W^{(\nu)} \prod_{j=1}^{\nu} U_j^{(r-4)/2} \prod_{j=\nu+1}^{r-1} U_j^{(r-2)/2} U_r^{r/2} \in \mathsf Z (A^{(\nu)}) \,,
\]
and $z'$ is the only factorization of $A^{(\nu)}$ which is divisible by $U$. Therefore we obtain that
\[
\mathsf t (A^{(\nu)}, U) = \mathsf d (z,z') = |z'| = 2 + \nu (r-4)/2 + (r-1-\nu)(r-2)/2  + r/2 = 3- \nu + r (r-2)/2 \,.
\]
\qed[Proof of {\bf A3.}]

{\it Proof of}\, {\bf A4}.\, We set $e_{r+1}:= e_1$ and define
\[
U = (e_0+e_1+e_2)(e_0+e_2+e_3) \cdot \ldots \cdot (e_0+e_r+e_1) \,.
\]
Since $r$ is divisible by $4$, we infer that $U \in \mathcal A (G)$ and clearly we have $|U|=r$. For all $i \in [1,r]$ and all $\nu \in [1, r-2]$, we define
\[
\begin{aligned}
U_i = e_i^2, \quad  V_i & = (e_0+e_i+e_{i+1})e_i^{-1} e_{i+1}^{-1} \prod_{j=1}^r e_j , \quad  V_0^{(\nu)} = (e_0 + e_{r-1} + e_r) (e_1+ \ldots + e_{\nu}) \prod_{j=\nu+1}^{r-2} e_j ,  \\
 W^{(\nu)} & = (e_1 + \ldots + e_{\nu})e_1 \cdot \ldots \cdot e_{\nu},  \quad \text{and} \quad A^{(\nu)} = V_0^{(\nu)} \prod_{j \in [1,r] \setminus \{r-1\}} V_j \,.
\end{aligned}
\]
Then $V_0^{(\nu)}, W^{(\nu)}, V_1, \ldots, V_r, U_1, \ldots, U_r \in \mathcal A (G)$,
\[
z = V_0^{(\nu)} \prod_{j \in [1,r] \setminus \{r-1\}} V_j \in \mathsf Z (A^{(\nu)}), \quad z' = U W^{(\nu)} \prod_{j=1}^{\nu} U_j^{(r-4)/2} \prod_{j=\nu+1}^r U_j^{(r-2)/2} \in \mathsf Z (A^{(\nu)}) \,,
\]
and $z'$ is the only factorization of $A^{(\nu)}$ which is divisible by $U$. Therefore we obtain that
\[
\mathsf t (A^{(\nu)}, U) = \mathsf d (z,z') = |z'| = 2 + \nu (r-4)/2 + (r-\nu)(r-2)/2  = 2 - \nu + r (r-2)/2 \,.
\]
\qed[Proof of {\bf A4.}]

\smallskip
5. We suppose that $r \ge 4$ and will proceed by induction on $r$.  If $r$ is even, then \cite[Theorem 5.1]{Ga-Ge-Sc15a} implies that
\[
\mathsf t (H) = \mathsf t (G) = 1 + r^2/2 \,,
\]
and
\[
[2, \mathsf D (G)] \subset \Ta (G) \subset \Ta (H) \subset [2, \mathsf t (H)] = [2, \mathsf t (G)] = [2, 1 + r^2/2 ] \,.
\]
whence it remains to show that
\begin{equation} \label{even}
[r+2, r^2/2]= [\mathsf D (G)+1, r^2/2] \subset \Ta (G) \,.
\end{equation}
If $r$ is odd, then $2 + r(r-1)/2 \in \Ta (G)$ and by the induction hypothesis
\[
[2, 1 + (r-1)^2/2] = \Ta (C_2^{r-1}) \subset \Ta (G)
\]
whence it remains to show that
\begin{equation} \label{odd}
[2 + (r-1)^2/2, 1 + \frac{r(r-1)}{2}] \subset \Ta (G) \,.
\end{equation}

(i) Suppose that $r=4$. By \eqref{even},  it remains to show that $[6,8] \subset \Ta (G)$, and this follows from {\bf A1}.

(ii) Suppose that $r=5$. By \eqref{odd}, remains to show that $[10, 11] \subset \Ta (G)$, and this follows from {\bf A2.}

(iii) Let $r \ge 6$ be even. By the induction hypothesis, we have
\[
[2, 3+ r(r-3)/2] = [2, 2+ (r-1)(r-2)/2 ] \subset \Ta (C_2^{r-1}) \subset \Ta (G)
\]
whence it remains to show that
\[
[4 + r(r-3)/2, r^2/2] \subset \Ta (G) \,.
\]
By {\bf A1}, we infer that
\[
[2 + r(r-2)/2, 1 + r^2/2] \subset \Ta (G) \,,
\]
whence it remains to verify that
\[
[4 + r(r-3)/2, 1 + r(r-2)/2] \subset \Ta (G) \,.
\]
If $r$ is divisible by $4$, this follows from {\bf A4}. If $r$ is congruent $2$ modulo $4$, this follows from {\bf A3}.

(iv) Let $r \ge 7$ be odd. By \eqref{odd},  it remains to show that
\[
[2+  (r-1)^2/2, 1+ r(r-1)/2] \subset \Ta (G) \,.
\]
By {\bf A2}, the interval $[2 + (r-1)(r-2)/2, 1 + r(r-1)/2] \subset \Ta (G)$ whence the assertion follows.
\end{proof}

\medskip
\begin{proof}[Proof of Theorem \ref{5.2}]~
We may suppose that $H$ is reduced and we consider a divisor theory $H \hookrightarrow F = \mathcal F (P)$. Then $G = \mathsf q (F)/\mathsf q (H)$ is the class group of $H$, and for an  element $a \in F$, $|a|_F = |a| \in \N_0$ denotes the length of $a$ with respect to $F$.

1. and 3. Suppose that every class contains a prime divisor.
Theorem \ref{5.1}.2 shows that $\Ta (C_2 \oplus C_2) = \{3\}$ and Similarly, we can show that $\Ta (C_3) = \{3\}$. Lemma \ref{5.6} implies that $[3, \mathsf D (G)]$ resp. $\N_{\ge 3} \subset \Ta (G)$. If there is a nonzero class containing at least two prime divisors, then $2 \in \Ta (H)$ by Lemma \ref{5.3}.3. Since $\Ta (G) \subset \Ta (H)$ by Lemma \ref{5.3}.2, it remains to show that $2 \in \Ta (G)$ whenever $\mathsf D (G) \ge 4$.
We distinguish four cases.

(i) Suppose that $G$ contains an element $g$ of infinite order. Then $U = (2g)(-g)^2$ and $V = (-3g)(2g)g$ are atoms. Since  $\mathsf Z(UV) = \{U V,  (-g)g  (-3g)(2g)(2g)(-g)\}$, it follows that $\mathsf t (UV,U) = 2$.

(ii) Suppose that  $G$ contains an element $g$ with $\ord (g)=n \ge 4$.  Then $U=g^n$ and $V=(-2g)(2g)$ are atoms. Since $\mathsf Z(UV)=\{U V, g^{n-2}(2g)\cdot (-2g)g^{2}\}$, it follows  that $\mathsf t(UV, U)=2$.

(iii) If all elements of $G$ have order two, then $2 \in \Ta (G)$ by Theorem \ref{5.1}.4.

(iv) Suppose that  all elements of $G$ have order three, and let $e_1$ and $e_2$ be two independent elements of order $3$. Then  $U=(e_1+e_2)^3$ and $V=(e_1+e_2)e_1^2e_2^2$ are atoms. Since
$\mathsf Z(UV)=\{U V, e_1e_2(e_1+e_2)^2\cdot e_1e_2(e_1+e_2)^2\}$, it follows that $\mathsf t(UV, U)=2$.

\smallskip
2. Suppose that every class contains at least  $\mathsf D(G)+1$ prime divisors. We  start with the following assertion.

\begin{enumerate}
\item[{\bf A.}\,] For every $a\in H$ and every $u\in \mathcal A(H)$ with $\mathsf t(a,u)\ge \mathsf D(G)+1$  there exists $b\in H$  with $|b| < |a|$ and $\mathsf t(b,u)\ge \mathsf t(a, u)-1$.
\end{enumerate}

{\it Proof of}\, {\bf A}.\, Let $a \in H$ and $u \in \mathcal A (H)$ with $\mathsf t (a,u) \ge \mathsf D (G)+1$. By definition of $\mathsf t (a,u)$,
there exists a factorization $z=u_1\cdot\ldots\cdot u_k \in \mathsf Z (a)$, where $k \in \N$ and $u_1,\ldots, u_k \in \mathcal A(H)$, such that for all factorization $z' \in \mathsf Z (a) \cap u \mathsf Z (H)$, we have $\mathsf d(z,z')\ge \mathsf t(a,u)$.

For every atom $v \in H$,     we have $|v| \in [1, \mathsf D(G)]$ (\cite[Theorem 5.1.5]{Ge-HK06a}) and $v$ is a prime in $H$ if and only if $|v|=1$. Since $\mathsf t (a,v)=0$ for all primes $v \in H$, we infer that  $u$ is not a prime. If $v$ is a prime dividing $a$, then $\mathsf Z (a) = v \mathsf Z (v^{-1}a)$ whence $\mathsf t (a, u) = \mathsf t (a v^{-1}, u)$. Thus we may assume without restriction that $a$ is not divisible by any prime of $H$. This implies that $\max \mathsf L (c) \le |c|/2$ for every divisor $c$ of $a$.

Suppose there exists an $i \in [1,k]$, say $i=1$, such that $u \t_F u_2 \cdot \ldots \cdot u_k$. We define $b=u_2\cdot\ldots\cdot u_k$ and clearly we have $|b| < |a|$.
Assume to the contrary that $\mathsf t(b,u)\le \mathsf t(a, u)-2$. Then for $z_1 = u_2 \cdot \ldots \cdot u_k \in \mathsf Z (b)$ there exists a factorization $z_1' = u w_1 \cdot \ldots \cdot w_s$, with $s \in \N$ and $w_1, \ldots, w_s \in \mathcal A (H)$, such that $\mathsf d (z_1, z_1') \le \mathsf t (a,u)-2$. Setting $z'=u_1u w_1 \cdot \ldots \cdot w_s \in \mathsf Z (a)$ we obtain that $\mathsf d (z,z') \le \mathsf d (z_1, z_1') \le \mathsf t (a,u)-2$, a contradiction.

Now we suppose that $u$ does not divide (in $F$) any subproduct of $u_1 \cdot \ldots \cdot u_k$. This implies $k \le |u|$.
If $|au^{-1}| \le k$, then
\[
\mathsf t(a,u)\le \max \{k, 1+ |au^{-1}|/2\} \le k \le |u| \le \mathsf D(G) \,,
\]
a contradiction.
Thus $|au^{-1}| \ge k+1$. Therefore there exist $i\in [1,k]$, say $i=1$, and $p_1, p_2\in P$ with $p_1p_2\mid_F u_1$ such that $u\mid_F a(p_1p_2)^{-1}$. If $u_1=p_1p_2$, then $u \mid_F u_2 \cdot \ldots \cdot u_k$, a contradiction.

Thus  $u_1 \ne p_1p_2$.
Since every class contains at least  $\mathsf D(G)+1$ prime divisors and $|u| \le \mathsf D (G)$, there exists a  $p\in P$ with $p\nmid_F u$ and $[p]=[p_1]+[p_2]$. We define $u_1'=u_1p(p_1p_2)^{-1}$ and $b=ap(p_1p_2)^{-1}$.  Then $u_1'\in \mathcal A(H)$ and $b \in H$ with $|b| <|a|$. Assume to the contrary that $\mathsf t(b,u)\le \mathsf t(a, u)-2$. Then for $z_2=u_1' u_2 \cdot \ldots\cdot u_k \in \mathsf Z (b)$  there exists a factorization $z_2'=u w_1\cdot \ldots \cdot w_s \in \mathsf Z (b)$ such that $\mathsf d(z_2, z_2')\le \mathsf t(a,u)-2$, where $s\in \N$ and $w_1,\ldots, w_s\in \mathcal A(H)$.
Since $p\nmid_F u$, we have $u_1'\neq u$.  If there exists $j\in [1,s]$ such that $w_j=u_1'$, then $z^*=u w_1\cdot\ldots \cdot w_{j-1} u_1 w_{j+1}\cdot\ldots\cdot w_s$ is a factorization of $a$ and $\mathsf d(z,z^*)=\mathsf d(z_2, z_2')\le \mathsf t(a, u)-2$, a contradiction. Thus $u_1'\not\in \{u, w_1,\ldots, w_s\}$. Since $p \mid_F b$ but $p\nmid_F u$, there exists $i\in [1,s]$, say $i=1$, such that $p\mid_F w_1$ and $w_1\not\in \{u_1', u_2 ,\ldots, u_k\}$.
Then $w_1'=w_1p_1p_2p^{-1}$ is a product of at most two atoms, say $w_1'=v_1 v_2$, where $v_1\in \mathcal A(H)$ and $v_2\in \mathcal A(H)\cup \{1_H\}$, and  $z'=u v_1 v_2 w_2\cdot\ldots\cdot w_s \in \mathsf Z (a)$ with $\mathsf d(z, z')\le \mathsf d(z_2, z_2')+1\le \mathsf t(a,u)-1$, a contradiction.
 \qed[Proof of {\bf A.}]

\medskip
By 1., we have $[2, \mathsf D (G)] \subset \Ta (H)$ whence the assertion follows if  $\mathsf t (H)\le \mathsf D(G)$.
We suppose that $\mathsf t(H)\ge \mathsf D(G)+1$ and have to verify that $[\mathsf D (G), \mathsf t (H)] \subset \Ta (H)$.
Let $a\in H$ and $u\in \mathcal A(H)$ such that $\mathsf t(a,u)=\mathsf t(H)$.
Since $\mathsf t(b,v)\le \max \mathsf L (b) \le |b|$ for all $b\in H$ and $v\in \mathcal A(H)$,  {\bf A} implies that there exist $a=b_0,\ldots, b_k \in H$ with $k\in \N$,  $|b_0|>\ldots>|b_k|$, and $\mathsf t(b_{i+1},u)\ge \mathsf t(b_i,u)-1$ for every $i\in [1, k-1]$ such that $\mathsf t(b_k, u)\le \mathsf D(G)$.
Therefore $[\mathsf D(G), \mathsf t(H)] \subset \{\mathsf t(b_i, u)\mid i\in [0,k ]\} \subset \Ta(H)$.
\end{proof}

\providecommand{\bysame}{\leavevmode\hbox to3em{\hrulefill}\thinspace}
\providecommand{\MR}{\relax\ifhmode\unskip\space\fi MR }
\providecommand{\MRhref}[2]{%
  \href{http://www.ams.org/mathscinet-getitem?mr=#1}{#2}
}
\providecommand{\href}[2]{#2}

\end{document}